\documentclass{amsart}
\usepackage{amsfonts, amssymb, graphicx}
\newtheorem{thm}{Theorem}[section]
\newtheorem{cor}[thm]{Corollary}
\newtheorem{lem}[thm]{Lemma}
\newtheorem{eg}[thm]{Example}

\newtheorem{prop}[thm]{Proposition}

\theoremstyle{definition}
\newtheorem{defn}[thm]{Definition}
\theoremstyle{remark}

\newcommand{\comments}[1]{}

\mathchardef\mhyphen="2D

\newcommand{\gt}{t}
\newcommand{\lvl}{\text{lvl}}

\newcommand{\ho}{K}
\newcommand{\co}{{Ch}}
\newcommand{\gr}{{gr}}
\newcommand{\colf}{{Ch^{lf}}}
\newcommand{\cof}{{Ch^f}}
\newcommand{\dg}{{DG }}

\newcommand{\cmd}{{\text{Mod}_\infty}}

\newcommand{\nucmd}{{\text{Mod}_\infty^{nu}}}

\newcommand{\pmd}{\text{mod}}
\newcommand{\pcmd}{{\text{mod}_\infty}} 

\newcommand{\ccm}{{\text{cmod}_\infty}}

\newcommand{\susp}{s}
\newcommand{\cl}[1]{\left< {#1} \right>}

\newcommand{\tns}[1]{\otimes^{[#1]}}
\newcommand{\ts}[1]{\stackrel{\infty}{\otimes}_{#1}}
\newcommand{\inthom}{{\mathcal{H}\textit{om}}}
\newcommand{\yoneda}[2]{h_{#1} (#2)}
\newcommand{\yoned}[1]{h_{#1}}
\newcommand{\Hom}{\text{Hom}}

\newcommand{\abr}{\overline{\mathcal{B}}}
\newcommand{\acbr}{\mathcal{B}}

\newcommand{\acbn}[1]{\mathcal{B}_{#1}}

\newcommand{\br}{\overline{\mathbf{B}}}
\newcommand{\cbr}{\mathbf{B}}

\newcommand{\cbn}[1]{\mathbf{B}_{#1}}

\newcommand{\mult}{m}
\newcommand{\moduledata}{{(\mathfrak{a}, \mathfrak{b})}}

\newcommand{\labs}{{\mathfrak{s}}}
\newcommand{\glue}{{\mathfrak{t}}}
\newcommand{\gluecyc}{{\mathfrak{r}}}

\newcommand{\Unit}{\mathbf{U}}
\newcommand{\Diagonal}{\mathbf{D}}

\newcommand{\unit}{u} 

\newcommand{\ov}[1]{\overline{#1}}
\newcommand{\one}{\mathbf{1}}

\newcommand{\Z}{\mathbb{Z}}

\newcommand{\C}{\mathbb{C}}

\newcommand{\N}{\mathbb{N}}
\newcommand{\p}{\mathbb{P}}
\newcommand{\K}{\mathbb{K}}

\newcommand{\mca}{\mathcal{A}}
\newcommand{\mcb}{\mathcal{B}}
\newcommand{\mcc}{\mathcal{C}}
\newcommand{\mcd}{\mathcal{D}}

\newcommand{\mcf}{\mathcal{F}}
\newcommand{\mcg}{\mathcal{G}}

\newcommand{\mci}{\mathcal{I}}

\newcommand{\mcl}{\mathcal{L}}
\newcommand{\mco}{\mathcal{O}}

\newcommand{\mct}{\mathcal{T}}

\newcommand{\mba}{\mathbf{a}}

\usepackage{tikz, tikz-cd}
\usepackage{hyperref}

\title{Orlov spectra as a filtered cohomology theory}
\author[L. Katzarkov]{Ludmil Katzarkov}
\address{Fakult\"at f\"ur Mathematik , Universit\"at Wien, 1090 Wien, Austria}
\email{ludmil.katzarkov@univie.ac.at}

\author[G. Kerr]{Gabriel Kerr}
\address{Department of Mathematics, University of Miami, Coral Gables, FL, 33146, USA}
\email{gdkerr@math.miami.edu}

\begin{document}

\maketitle

\begin{abstract}
This paper presents a new approach to the dimension theory of triangulated categories by considering invariants that arise in the pretriangulated setting.
\end{abstract}

\section{Introduction}

In \cite{Rouquier}, Rouquier gave several results on the dimension theory of triangulated categories. Following this paper, Orlov computed the dimension of the derived category of coherent sheaves on an elliptic curve and found it to equal one in \cite{Orlov}. Orlov then advanced a more general perspective on dimension theory by defining the spectrum of a triangulated category, now called the Orlov spectrum, which includes the generation times of all strong generators. This important concept serves as a more nuanced invariant than dimension, as it gives one a way to compare all strong generators. Many results on Orlov spectra are obtained in \cite{BFK}, where they observe, in certain instances, gaps in the generation times. On the other hand, there are many cases where no such gaps appear. This lead to several outstanding conjectures on the existence of such gaps and their potential relationship to the Hodge conjecture and homological mirror symmetry.

While the triangulated setting serves as an accessible model for homological invariants, it is generally accepted that triangulated categories are inadequate for giving a natural characterization of homotopy theory for derived categories. Instead of working in this setting, it is advisable to lift to a pretriangulated category, or $(\infty, 1)$-category framework, where several constructions are more natural \cite{Lurie, Drinfeld}. In this paper, we study the dimension of triangulated categories by lifting to pretriangulated \dg or $A_\infty$-categories. 

When the category $\mct$ is strongly generated by a compact object $G$, we upgrade several classical results in dimension theory of abelian categories to the pretriangulated setting and find that the natural filtration given by the bar construction plays a determining role in the calculus of dimension. Indeed, if $G$ is such a generator, using a result of Lef\`evre-Hasegawa, we can regard $\mct$ as the homotopy category of perfect modules over an $A_\infty$ algebra $A_G = \Hom^* (G,G)$. In addition to being a \dg category, the category of perfect $A_\infty$ modules over $A_G$ is enhanced over filtered chain complexes, where the filtration is obtained through the bar construction. This filtration descends to the triangulated level. The first main result, Theorem \ref{thm:main1}, in this paper is that the generation time of a strong generator $G$ equals the maximal length of this filtration.

\begin{thm} The generation time of $G \in \mct$ equals the supremum over all $M, N \in A_G \mhyphen \pcmd$ of the lengths of $\Hom_{A_G \mhyphen \pcmd} (M, N)$ with respect to the filtration induced by the bar construction. 
\end{thm}

As a result, we develop a filtered cohomology theory which yields the generation times that occur in Orlov spectra.  The lengths referred to in this theorem are those of the filtrations induced on the cohomology of the complexes, or the $\text{Ext}$ groups, by the pretriangulated filtrations. In practice, it is possible to compute these lengths by calculating their spectral sequences which will converge under very mild assumptions.

Another filtration that occurs naturally from the bar construction is on the tensor product. This filtration is especially useful as one may define change of base as a tensor product with an appropriate bimodule. After establishing basic adjunction results in the next section, we generalize the classical change of base formula for dimension to the $A_\infty$ algebra setting in Theorem \ref{thm:main2}. A new multiplicative constant appears in this version which is related to the speed at which a spectral sequence associated to the tensor product filtration converges. 

\begin{thm} Let $P$ be a $(B, A)$-bimodule and $M$ a left $A$-module. Suppose the
spectral sequence of $P \ts{A} M$ degenerates at the $(s + 1)$-st
page. If the convolution functor $P \ts{} \_$ is faithful, then
\begin{equation*} \textnormal{lvl}_A (M) \leq \textnormal{lvl}_A (P) + s \cdot \textnormal{lvl}_B (P \ts{A} M ) \end{equation*} 
\end{thm}

Here $\lvl_A (M)$ plays the role of homological, or projective, dimension of a module $M$. If the algebra $A$ is formal, the constant $s$ is $1$ and we see the classical formula. If higher products are relevant, one must modify the classical inequality. \\

{\em Acknowledgements:} Support was provided by NSF Grant DMS0600800, NSF FRG Grant DMS-0652633, FWF
Grant P20778, and an ERC Grant --- GEMIS. Both authors would like to thank Matt Ballard, Colin Diemer, David Favero, Maxim Kontsevich, Dima Orlov, Pranav Pandit, Tony Pantev and Paul Seidel for helpful comments and conversation during the preparation of this work. The second author would like to thank David Favero in particular for his patient explanations of results on the dimension theory of triangulated categories. 

\section{$A_\infty$ Constructions}

This section will review many definitions and constructions related to $A_\infty$ algebras and modules. The aim of our treatment is to approach this subject with a special emphasis on the filtrations arising from the bar constructions. These filtrations are the main technical structure we use in the dimension theory for pretriangulated categories.

After reviewing some standard definitions, we will give the definitions of filtered tensor product, filtered internal Hom and duals in the category of $A_\infty$-bimodules. The mantra that all constructions in the $A_\infty$ setting are derived constructions will be continually reinforced. Moreover, the above functors will land in the category of lattice filtered $A_\infty$-modules, which preserves the relevant data for a study of dimension. The $\otimes-\inthom$ adjunction, usually written in either the abelian or derived setting, will be formed as an adjunction between filtered \dg functors. The categorical formulation of this statement is that the category $Alg_\infty$ is a biclosed bicategory enriched over filtered cochain complexes. We will utilize this to update classical results on the relationship between flat and projective dimensions for perfect modules. 

\subsection{Fundamental Notions} \label{subsub:lf}

We take a moment to lay out some basic notation
and fix our sign conventions. All algebras and vector spaces will be over a fixed field $\K$ and categories will be $\K$-linear categories. Let $\gr$
be the category of graded vector spaces over $\K$ and finite sums of homogeneous maps. We
take $\co$ to be the category of
cochain complexes of vector spaces over $\K$ and finite sums of homogeneous maps.
We will identify $\Hom_\co$ with the internal $\Hom$ whose differential of
\begin{equation*} f \in \Hom_\co^k ((C, d_C), (C^\prime, d_{C^\prime} )) \end{equation*}
is the usual one, namely, \begin{equation*} df := f \circ d_C -
(-1)^k d_{C^\prime} \circ f \end{equation*} Finally, we take $\ho$
to be the category of chain complexes and cochain maps. In other
words, maps which are cocycles relative to $d$ in $\co$.

In the above definition, we are purposefully vague with respect to
the abelian group indexing the grading. For most of the paper, we
will assume our chain complexes are $\Z$-graded, but there will be
examples of the $(\Z / 2 \Z)$-graded case. This should cause no
difficulty as the proofs will be independent of this choice.

We view $\co$ as a closed category with respect to the
tensor product along with the Koszul sign rule $\gamma_{V, W} : V
\otimes W \to W \otimes V$ given by:

\begin{equation} \label{eq:symmon} \gamma_{V, W} (v \otimes w) = (-1)^{|v| |w| } w \otimes v \end{equation}

We will need to implement this sign convention when discussing
tensor products of maps as well. For this we follow the usual convention. Namely, given  homogeneous maps $f \in \Hom_\gr^* (V_1 , V_2)$   $g \in \Hom_\gr^* (W_1, W_2)$ then we define $f \otimes g \in \Hom_\gr (V_1 \otimes W_1 , V_2 \otimes W_2 )$ via $(f \otimes g) (v \otimes w) = (-1)^{|g||v|} f (v) \otimes g(w)$. 


By a differential graded, or \dg , category $\mcd$ we mean a category enriched in $\co$. We let 
\begin{equation*} \yoned{\_} : \mcd \to \co^{\mcd^{op}} \end{equation*}
be the Yoneda functor given by $\yoneda{E}{E^\prime} = \Hom_\mcd (E^\prime , E )$. 

In categories $\gr$, $\co$ and $\ho$, we have the shift functor $\susp$ which sends $V^*$ to $V^{* + 1}$. On morphisms we have $\susp (f) = (-1)^{|f|} f$.  There is also a (degree $1$) natural transformation $\sigma : I \to \susp$ defined as $\sigma (v) = (-1)^{|v|} v$. One can utilize $\sigma$ to translate the signs occurring in various bar constructions given
in this text and those in the ordinary desuspended case. In particular, given a map $f : V^{\otimes n} \to W^{\otimes m}$ in $\co$ we define $\susp_\otimes (f) : (\susp V)^{\otimes n} \to (\susp W)^{\otimes m}$ to be $ \sigma^{\otimes m} \circ f \circ (\sigma^{-1})^{\otimes n}$. We will often use this notation to write the equations defining various structures
without mentioning the elements of our algebras or modules. A nice account of the various choices and techniques used in sign
conventions can be found in \cite{Deligne1}.

 Filtrations will occur throughout this paper and our initial approach will be rather general. We partially order $\Z^k$ for any $k \in \N$ with the product order. A  lattice filtered complex will consist of the data $\mathbf{V} = \left(V, \{V_\alpha\}_{\alpha \in \Z^k}\right)$ for some $k \in \N$, where $V$ is an object in $\co$ and $\{V_\alpha\}_{\alpha \in \Z^k}$ is a collection of subcomplexes partially ordered by inclusion. If $k = 1$, we simply call $V$ filtered. Given two lattice filtered complexes $\mathbf{V} = \left(V, \{V_\alpha\}_{\alpha \in \Z^k}\right)$ and $\mathbf{W} = \left(W, \{W_\beta \}_{\beta \in \Z^l}\right)$, we define the lattice filtered tensor product and internal hom as follows. 
\begin{equation*} \mathbf{V} \otimes \mathbf{W} = \left(V \otimes W , \{ V_\alpha \otimes W_\beta \}_{\left(\alpha , \beta \right) \in \Z^{k + l}} \right) \end{equation*}
and
\begin{equation*} \inthom \left(\mathbf{V}, \mathbf{W} \right) = \left(\inthom \left(V, W\right) , \{\inthom_{- \alpha , \beta } \left(V , W\right) \}_{\left(\alpha , \beta \right) \in \Z^{k + l}} \right) \end{equation*}
where $\inthom_{-\alpha , \beta} \left( V, W \right) = \{\phi : V \to W | \phi ( V_\alpha ) \subseteq W_\beta \}$. The category of lattice filtered complexes and filtered complexes will be denoted $\colf$ and $\cof$ respectively. We note that the above constructions make $\colf$ a closed symmetric monoidal category. 

Given a \dg category $\mcd$, we define the category $\mcd^{lf}$ to have objects consisting of the data $\mathbf{E} = \left( E, \{E_\alpha \}_{\alpha \in \Z^k} \right)$ where $\left( \yoneda{E}{E^\prime} , \{\yoneda{E_\alpha}{E^\prime} ) \}_{\alpha \in \Z^k }  \right) \in \colf$ for every object $E^\prime \in \mcd$. The cochain complex of morphisms between $\mathbf{D}$ and $\mathbf{E}$ is simply $\Hom_\mcd (D, E)$. Restricting to the case of $k=1$ yields the definition of $\mcd^f$. 

The total filtration functor $Tot : \colf \to \cof$ is defined as 
\begin{equation*} Tot \left( V , \{V_\alpha \}_{\alpha \in \Z^k}  \right) = \left( V , \{ \cup_{|\alpha | = n} V_\alpha \}_{n \in \Z} \right)\end{equation*}
for $k \ne 0$ where $|\alpha | = a_1 + \cdots + a_k$ for $\alpha = (a_1, \ldots, a_k)$. One needs to deal with $k = 0$ a bit differently and define $Tot \left( V, \{V_0 \}\right) = \left( V, \{V^\prime_n \} \right)$ with $V^\prime_n = 0$ for $n < 0$ and $V^\prime_n = V_0$ otherwise.

Now suppose $\mathbf{V} \in \cof$ is a filtered complex. Letting $Z_n$ be the subspace of cocycles in $V_n$, we have that the cohomology 
\begin{equation*} H^* ( \mathbf{V} )  = ( H^* (V) , \{ H^* (V)_n =  Z_n / im (d) \cap Z_n  \}_{n \in \Z} )\end{equation*}
is then a filtered object in $\gr$. We define the upper and lower length of the filtration as follows. If $\cup_n H^* (\mathbf{V} ) \ne H^* (V)$ we take $\ell_+ (\mathbf{V} ) = \infty$ and if $\cap_n H^* (\mathbf{V} ) \ne 0$ then $\ell_- (\mathbf{V} ) = - \infty$. Otherwise, we define these lengths as
\begin{equation} \label{eq:length}
\ell_+ (\mathbf{V} ) = \inf \{n : H^* (V)_n = H^* (V) \} \hspace{.3in} \ell_- (\mathbf{V}) = \inf \{ n : H^* (V)_n \ne 0 \} .
\end{equation}
By the length $\ell (\mathbf{V})$ of $\mathbf{V}$ we will mean the maximum of $|\ell_+ (\mathbf{V} ) |$ and $| \ell_- (\mathbf{V} )|$. We extend these definitions to $\mathbf{V} \in \colf$ by taking length of $Tot (\mathbf{V} )$. 

Given a \dg category $\mcd$ and an object $\mathbf{E}  \in \mcd^{lf}$, we define the lengths of $\mathbf{E} $ as 
\begin{eqnarray*}
\ell_+ (\mathbf{E} ) & = & \sup \{ \ell_+ (\yoneda{\mathbf{E}}{E^\prime}) : E^\prime \in \mcd \} , \\ 
\ell_- (\mathbf{E} ) & = & \inf \{ \ell_- (\yoneda{\mathbf{E}}{E^\prime} ): E^\prime \in \mcd \} , \\ 
\ell (\mathbf{E} ) & = & \sup \{ \ell (\yoneda{\mathbf{E}}{E^\prime}) : E^\prime \in \mcd \} .
\end{eqnarray*}
Given two \dg categories $\mcd$, $\tilde{\mcd}$, a \dg functor $F : \mcd \to \tilde{\mcd}^f$ and $E \in \mcd$, we take $\ell_\pm^F (E ) = \ell_\pm (F (E))$ and $\ell^F = \sup \{\ell^F (E) : E \in \mcd \}$. One can consider $\ell^F$ as a generalization of the cohomological dimension of a functor between abelian categories. Note that in the \dg category $\co$ the two notions of length are equal. In other words, the definition given by equations \ref{eq:length} yield the same quantities as the definition above using the Yoneda embedding $\yoned{\_}$.

A motivating example for the above definitions is the case where $\mcd$ and $\tilde{\mcd}$ are categories of bounded below cochain complexes of injective objects in abelian categories $D$ and $\tilde{D}$. Note that these categories admit embeddings into their filtered versions by sending any complex $E^*$ to $\left( E^* , \{ \tau_n (E^* ) \}_{n \in \Z}\right) $ where $\tau_n (E^* )^k = E^k$ for $k \leq n$ and zero otherwise. Assuming $D$ and $\tilde{D}$ have enough injectives, any functor $F : D \to \tilde{D}$ has the (pre)derived \dg functor $RF : \mcd \to \tilde{\mcd}$ and after composition with the embedding above one has a \dg functor $\mcf : \mcd \to \tilde{\mcd}^f$. It is then plain to see that $\ell^{\mcf}$ equals the cohomological dimension of $F$.

\subsection{$A_\infty$-algebras}

One of the fundamental structures in our study is an $A_\infty$-algebra.

\begin{defn} A non-unital $A_\infty$-algebra $A$ is an object $A \in \co$ and a collection of degree $1$ maps
$\mu_A^n : (\susp A)^{\otimes n} \to \susp A$ for $n > 0$ satisfying
the relation \begin{equation*} \sum_{k = 0}^n \left[ \sum_{r =
0}^{n - k} \mu_A^{n - k + 1 } \circ (\one^{\otimes r} \otimes
\mu_A^{k} \otimes \one^{\otimes (n - r - k)} ) \right] = 0
\end{equation*} for every $n$. \end{defn}

We note that it is common to see the definition utilizing the desuspended maps $\susp^{-1}_\otimes (\mu_A^n )$ which involves more intricate signs. 

In this paper we will assume that our $A_\infty$-algebras come
equipped with a strict unit. We recall that this means there
exists a unit map $\unit : \K \to A[1]$ where
\begin{eqnarray}\label{eq:unit1} \mu_A^2 (\unit \otimes \one ) = \one =  - \mu_A^2 (\one
\otimes \unit ) & & \\ \label{eq:unit2} \mu^n ( \one^{\otimes r} \otimes \unit
\otimes \one^{\otimes (n - r - 1)}) =  0 & & \hspace{.4in}
\hbox{for } n \ne 2 \end{eqnarray} We will normally write $e_A$ for $\unit (1)$ (or $e$ if the algebra is
implicit).

If $A$ is an $A_\infty$-algebra, we take $A^{op}$ to be the algebra with structure maps $\mu_{A^{op}}^k = \mu_A^k \circ \sigma_k$ where $\sigma_k : (\susp A)^{\otimes k} \to (\susp A)^{\otimes k}$ reverses the ordering of the factors via the symmetric monoidal transformation $\gamma$ in \ref{eq:symmon}.

It is immediate that the cohomology $H^* (A)$ defined with respect to $\mu^1_A$ is a graded
$\K$-algebra with multiplication induced by $\mu_A^2$. However,
the higher products determine more structure than the
cohomology algebra can express on its own. In order to see this we
need to be able to compare two different algebras. A homomorphism
of $A_\infty$-algebras is defined as follows.

\begin{defn} If $(A, \mu^*_A)$ and $(B, \mu^*_B)$ are $A_\infty$-algebras then a collection of graded maps $\phi^n : (\susp A)^{\otimes n} \to \susp B$ for $n \geq 1$ is an $A_\infty$-map if
\begin{eqnarray*} \sum_{k = 1}^n \left[ \sum_{r = 1}^{n - k} \phi^{n - k + 1 } \circ (\one^{\otimes r} \otimes \mu_A^{k} \otimes \one^{\otimes (n - r - k)} ) \right]
= \\ \indent \indent  \indent \indent \sum_{j = 1}^n \left[
\sum_{i_1 + \cdots + i_j = n} \mu_B^j \circ (\phi^{\otimes i_1}
\otimes \cdots \otimes \phi^{\otimes i_j} ) \right]
\end{eqnarray*}
\end{defn}

A strictly unital homomorphism is also required to preserve the unit as well as satisfying the identities 
\begin{equation*}  \phi^{r + s + 1} \circ ( \one^{\otimes r} \otimes \unit \otimes \one^{\otimes s}) = 0 \end{equation*} 
for all $r + s > 0$. The category of unital and non-unital $A_\infty$-algebras will be denoted $Alg_\infty$ and $Alg_\infty^{nu}$ respectively.

When all maps $\phi^k = 0$ except $\phi^1$, we call $\{\phi^k\}$
strict. If there is an $A_\infty$-map $\epsilon_A : A \to \K$ we
will call $A$ augmented.  Any augmented, strictly unital $A_\infty$ algebra is
required to satisfy the equation $\epsilon_A \unit = 1_\K$. 

It is important to observe that $[\phi^1]$ induces an algebra homomorphism $H^*
(A)$ to $H^* (B)$ so that cohomology is a functor from
$A_\infty$-algebras to ordinary algebras. When the induced map
$[\phi^1]$ is an isomorphism, we call $\phi^*$ a
quasi-isomorphism. The following proposition can be found in any
of the basic references given above.

\begin{prop} Given a quasi-isomorphism $\phi^* : A \to B$ there exists a quasi-isomorphism $\psi^* : B \to A$ for which $[\phi^1]$ and $[\psi^1]$ are inverse. \end{prop}

Some of the $A_\infty$-algebras discussed in this paper satisfy additional conditions.

\begin{defn} i) An $A_\infty$-algebra is formal if it is quasi-isomorphic to its cohomology algebra. \\ ii) An $A_\infty$-algebra is compact if its cohomology algebra is finite dimensional. \end{defn}

While it is rarely the case that an $A_\infty$-algebra is formal, there is an $A_\infty$-structure on its cohomology, called the minimal model, which yields a quasi-isomorphic $A_\infty$-algebra. It is a well known fact that, for $(A, \mu^*_A)$, this is a uniquely defined $A_\infty$-structure $(H^* (A), \tilde{\mu}_A^*)$ with $\tilde{\mu}_A^1 = 0$  (here $\tilde{\mu}_A^2 = [\mu_A^2]$ and the higher $\tilde{\mu}_A^*$ are determined by a tree level expansion formula). 
Let us state this as a proposition.

\begin{prop} For any $A_\infty$-algebra $(A, \mu_A^*)$ there is an $A_\infty$-algebra $(H^* (A), \tilde{\mu}_A^*)$, uniquely defined up to $A_\infty$-isomorphism, and a quasi-isomorphism $\phi_A : A \to H^* (A)$. We call $(H^* (A), \tilde{\mu}_A^*)$ the minimal model of $(A, \mu_A^*)$. When there exists a minimal model with $\tilde{\mu}_A^k = 0$ for all $k > 2$, we call $A$ formal. \end{prop}
It will be important to have at our disposal another equivalent
definition, the algebra bar construction, for which we closely follow \cite{Lefevre} and
\cite{Keller}. First, given $V \in \gr$ we denote the tensor
algebra and coalgebra by $T^a V$ and $T^c V$ respectively. As graded vector spaces, both are equal to
\begin{equation*} T V = \bigoplus_{n = 0}^\infty
V^{\otimes n} . \end{equation*} 
For space considerations, we will use bar notation and write $[v_1| \cdots |v_n]$, or simply $\mathbf{v}$, for $v_1 \otimes
\cdots \otimes v_n$  for an arbitrary element of $TV$. 
These spaces are bigraded, with one grading denoting
the length of a tensor product, and the other denoting the total
degree. Our notation conventions for these gradings will be
\begin{equation*} (T V)^{r, s} = \left\{[v_1 | \cdots |v_r] : \sum
|v_i| = s \right\} .
\end{equation*}
In many situations, we will be interested only in the length
grading, in which case we use the notation 
\begin{equation*} (T V)_{n} = \oplus_{k = 0}^n (T V)^{k, \bullet} \hspace{.4in} (T V)_{> n} = \oplus_{k > n} (T V)^{k , \bullet} . \end{equation*}

The algebra map for $T^a V$ is the usual product and the coalgebra map $\Delta : T^c V \to T^c V \otimes_\K T^c V$ is defined as
\begin{equation*} \Delta [v_1 | \cdots | v_n] = \sum_{i = 0}^n [v_1| \cdots |v_i] \otimes [v_{i + 1}| \cdots | v_n] \end{equation*} where the empty bracket $[]$ denotes the identity in
$\K$. 

The tensor coalgebra naturally lives in the category of
coaugmented, counital, dg coalgebras $Cog^\prime$. The objects in
this category consist of data $(C, d, \eta, \epsilon)$ where $C$
is a coalgebra, $d$ is a degree $1$, square zero, coalgebra
derivation, $\eta: C \to \K$ and $\epsilon: \K \to C$  are the counit and coaugmentation satisfying
$\eta \epsilon = 1_\K$. However, this category is too large for
our purposes and we instead consider a subcategory $Cog$
consisting of cocomplete objects. To define these objects,  take
$\pi: C \to \ov{C} = C / \K$ to be the cokernel of $\epsilon$.
Consider the kernel $C_{n} = \ker (\tilde{\Delta}^n )$ where
\begin{equation*}\tilde{\Delta}^n : C
\stackrel{\Delta^n}{\longrightarrow} C^{\otimes n}
\stackrel{\pi^{\otimes n}}{\longrightarrow} \ov{C}^{\otimes n} .
\end{equation*} 
Elements of $C_{n}$ are called $n$-primitive and $C_{1}$ is referred to as the coaugmentation ideal.
They form an increasing sequence 
\begin{equation*} C_{0} \subset C_{1} \subset \cdots .\end{equation*} 
This defines a natural inclusion $Cog^\prime \to (Cog^\prime)^f$ and we say that the augmented coalgebra $C$ is cocomplete if $C = \lim C_{n}$. One easily observes that the tensor coalgebra is an object of $Cog$ as
\begin{equation*} (T^c V)_n = (T V)^{n, \bullet} .\end{equation*}
Moreover, the tensor coalgebra $T^c V$ is cofree in the category $Cog$ (i.e. the tensor coalgebra functor is right
adjoint to the forgetful functor).

Now we recall the (coaugmented) bar functor
\begin{equation*} \acbr : Alg^{nu}_\infty \to Cog \end{equation*} 
which takes any non-unital $A_\infty$-algebra $(A, \mu_A)$ to $\acbr A = (T^c (\susp A) , b_A,
\eta_{\acbr A} , \epsilon_{\acbr A} )$. The definitions of the
counit and coaugmentation are clear. We define $b_A : T^c (\susp A) \to
T^c (\susp A)$ via its restriction to $(\susp A )^{\otimes n}$ as
\begin{equation*} b_A|_{(\susp A)^{\otimes n}} = \sum_{k = 1}^n \left(\sum_{r = 0}^{n - k}
\one^{\otimes r} \otimes \mu_A^{k} \otimes \one^{\otimes (n - r -
k)} \right) . \end{equation*}  
As it turns out, there is a model structure on $Cog$ and the
following theorem describes the essential image of $\acbr$ in
these terms.

\begin{thm}[\cite{Lefevre}, 1.3] The functor $\acbr$ is a full and faithful embedding of $Alg_\infty^{nu}$ into
the fibrant-cofibrant objects of $Cog$. Furthermore, minimal
models are sent to minimal models. \end{thm}

There are several variants of this construction, most importantly
the ordinary bar construction $\abr A = ({T} (\susp A)_{> 0}, b_A)$ which takes
values in cocomplete coalgebras. We fix notation for the inclusion to be
\begin{equation} \label{eq:barinc} \iota_A : \abr A \hookrightarrow \acbr A .\end{equation}
 The differential is simply the restriction of the one defined in the coaugmented case.
It is helpful to understand $\abr A$ when $A$ is an ordinary
algebra $A$. In this case, we see that $\abr A$ is just the
augmented bar resolution for $A$ (and hence, acyclic).


The bar construction of $A$ inherits the increasing filtration
\begin{equation*} \acbn{n} A := \oplus_{i \leq n} (\acbr A)^{i, \bullet} = (\acbr A)_{n} \end{equation*} We refer to this, and the module variants to come, as the length
filtration.

We note here that one advantage of the bar construction is the ease at which one can discuss structures that are more difficult to define in the category $Alg_\infty$. One example of this is the tensor product of two $A_\infty$-algebras $A, A^\prime$ which has more than one fairly intricate definition. In $Cog$ we define the tensor product of $\acbr A \otimes \acbr A^\prime$ in the usual way. We then say that $B \in Alg_\infty$ is quasi-isomorphic to the tensor product if $B = A \otimes A^\prime$ and $\acbr B$ is quasi-isomorphic to $\acbr A \otimes \acbr A^\prime$ in $Cog^f$. See \cite{Loday} for an article comparing various constructions of a natural quasi-isomorphism. 

\subsection{$A_\infty$ polymodules}

 We start this section with a general definition of a module over several $A_\infty$-algebras which we call a polymodule. It is both useful and correct to think of a polymodule as a bimodule with respect to the tensor product of several algebras or, even more simply, as a module over the tensor product of algebras and their opposites. This is analogous to defining an $(R, S)$ bimodule as opposed to an $R \otimes S^{op}$ module. We take this approach at the outset to avoid some of the cumbersome notation and uniqueness issues surrounding the tensor product of multiple $A_\infty$-algebras. This is accomplished utilizing the bar construction and working in the category of comodules where many structures are more accessible. One word of notational warning is that \cite{Lefevre} uses the different term polydules to define what we would call a module.

For this section, we fix $A_\infty$-algebras $A_1, \ldots, A_r$ and $B_1 , \cdots, B_s$ and write $\moduledata$ for the data $(A_1, \ldots, A_r | B_1 , \ldots, B_s )$. Let $P$ be a graded vector space and write
\begin{equation} \label{eq:poly} \cbr^\moduledata P = \acbr A_1 \otimes \cdots \otimes \acbr A_r \otimes P \otimes \acbr B_1 \otimes \cdots \otimes \acbr B_s  \end{equation}
for the bar construction of $P$.

When $\moduledata$ is fixed or understood from the context, we simply write $\cbr P$. We make a note that $\cbr P$ is naturally an object of $\co^{lf}$ where the lattice is $\Z^{r + s}$ and the filtration is induced by the length filtrations on the bar constructions. Given any $\gamma \in \Z^{r + s}$, we denote the $\gamma$ filtered piece of $\cbr P$ by $\cbr_\gamma P$. Observe also that $\cbr P$ is a cofree left comodule over the coalgebras $\acbr A_i$ and a cofree right comodule over coalgebras $\acbr B_i$ where $\Delta_{ i, P} : \cbr P \to \acbr A_i \otimes \cbr P$ and $\Delta_{P, j} : \cbr P \to \cbr P \otimes \acbr B_j$ are the comodule maps. These are defined by repeatedly applying $\gamma$ from equation \ref{eq:symmon} to permute the left factor of $\acbr A_i$ and right factor of $\acbr B_j$ to the left and right respectively, after having applied their comultiplications. We take,
\begin{equation*} \Delta_P = \Delta_{P, s} \circ \cdots \circ \Delta_{P, 1} \circ \Delta_{r, P} \circ \cdots \circ \Delta_{1, P} 
\end{equation*}
as the polymodule comultiplication from $\cbr P$ to $\acbr A_1 \otimes \cdots \otimes \acbr A_r \otimes \cbr P \otimes \acbr B_1 \otimes \cdots \otimes \acbr B_s $. The differentials on each coalgebra tensor to define the differential $d^\prime_P : \cbr P \to \cbr P$. 

\begin{defn} A non-unital $\moduledata = (A_1, \ldots, A_r | B_1, \ldots, B_s)$ polymodule $(P, \mu_P)$ is a graded vector space $P$ along with a degree $1$ map
\begin{equation} \label{eq:polymu} \mu_P : \cbr^\moduledata P \to P \end{equation}
satisfying the equation
\begin{equation} \label{eq:polymod} \mu_P  \circ [ (\one_{\acbr A_1} \otimes \cdots \otimes \one_{\acbr A_r} \otimes \mu_P \otimes \one_{\acbr B_1} \otimes \cdots \otimes \one_{\acbr B_s} ) \circ \Delta + d^\prime_P] = 0 .\end{equation}
We call the $P$ a polymodule if $\mu_P$ satisfies the unital conditions for every $i$ and $j$,
\begin{eqnarray*}
\mu_P \circ ( \epsilon_{\acbr A_1} \otimes \cdots \otimes \unit_{\acbr A_i} \otimes \cdots \otimes \epsilon_{\acbr A_r} \otimes \one_P \otimes \epsilon_{\acbr B_1} \otimes \cdots \otimes \epsilon_{\acbr B_r} ) & = & \one_P , \\ \mu_P \circ ( \epsilon_{\acbr A_1} \otimes  \cdots \otimes \epsilon_{\acbr A_r} \otimes \one_P \otimes \epsilon_{\acbr B_1} \otimes \cdots \otimes \unit_{\acbr B_j} \otimes \cdots \otimes \epsilon_{\acbr B_r} ) & = & \one_P , \\ \mu_P \circ ( \iota_{ A_1} \otimes \cdots \otimes \unit_{\acbr A_i} \otimes \cdots \otimes \iota_{ A_r} \otimes \one_P \otimes \iota_{ B_1} \otimes \cdots \otimes \iota_{ B_r}  ) & = & 0 , \\ \mu_P \circ (\iota_{ A_1} \otimes  \cdots \otimes \iota_{ A_r} \otimes \one_P \otimes \iota_{ B_1} \otimes \cdots \otimes \unit_{\acbr B_j} \otimes \cdots \otimes \iota_{ B_r} )  & = & 0 ,
\end{eqnarray*}
where $\epsilon$ is the coaugmentation and $\iota$ the inclusion from  \ref{eq:barinc}. 
\end{defn}

A bimodule is a polymodule for which $r = 1 = s$. A left module is a bimodule for which $B_1 = \K$ and similarly for a right module.  A non-unital morphism from the polymodule $P$ to $P^\prime$ is defined as any $\gr$ map $\phi : \cbr P \to P^\prime $. The collection of these maps forms a complex $\Hom_{\moduledata \mhyphen \nucmd} (P, P^\prime )$ with differential defined as
\begin{multline} \label{eq:polymor}
d \phi  = \phi \circ (\one_{\acbr A_1} \otimes \cdots \otimes \one_{\acbr A_r} \otimes \mu_P \otimes \one_{\acbr B_1} \otimes \cdots \otimes \one_{\acbr B_s}) \circ \Delta_{P}  \\ + \phi \circ d^\prime_P - (-1)^{|\phi |} \mu_{P^\prime} \circ (\one_{\acbr A_1} \otimes \cdots \otimes \one_{\acbr A_r} \otimes \phi \otimes \one_{\acbr B_1} \otimes \cdots \otimes \one_{\acbr B_s} ) \circ \Delta_P .
\end{multline}
A morphism $\phi \in \Hom_{\moduledata \mhyphen \cmd} (P, P^\prime )$ is a non-unital morphism satisfying the unital conditions
\begin{eqnarray*} \phi \circ ( \one_{\acbr A_1} \otimes \cdots \otimes \unit_{\acbr A_i} \otimes \cdots \otimes \one_{\acbr A_r} \otimes \one_P \otimes \one_{\acbr B_1} \otimes \cdots \otimes \one_{\acbr B_r} ) & = & 0 , \\ \phi \circ ( \one_{\acbr A_1} \otimes  \cdots \otimes \one_{\acbr A_r} \otimes \one_P \otimes \one_{\acbr B_1} \otimes \cdots \otimes \unit_{\acbr B_j} \otimes \cdots \otimes \one_{\acbr B_r} ) & = & 0 . \end{eqnarray*}
The fact that $\Hom_{\moduledata \mhyphen \cmd} (P, P^\prime )$ is indeed a subcomplex follows from the unital condition on algebras and modules. A morphism $\phi \in \Hom_{\moduledata \mhyphen \cmd} (P, P^\prime )$ is called strict if $\phi |_{\cbr_{> 0} P} = 0$.  A homomorphism is defined to be a cocycle in this complex. A homomorphism $\phi$ for which $\phi |_{\cbr_{\mathbf{0}} P} : P \to P^\prime$ is a quasi-isomorphism will be called a quasi-isomorphism. Given $\psi \in \Hom_{\moduledata \mhyphen \cmd} ( P^\prime, P^{\prime \prime} )$ we define composition as
\begin{equation*} \psi  \phi = \psi \circ (\one_{\acbr A_1} \otimes \cdots \otimes \one_{\acbr A_r} \otimes \phi \otimes \one_{\acbr B_1} \otimes \cdots \otimes \one_{\acbr B_s} ) \circ \Delta_P .\end{equation*}
It is a straightforward, albeit tedious check to see that these definitions make $\moduledata$ polymodules into a \dg category which we label $\moduledata \mhyphen \cmd$, or just $\cmd$. We write $H^0 (\cmd )$  ($H^* (\cmd )$) for the zeroth (graded) cohomology category .  The next proposition follows immediately from the discussion at the end of the previous section and the naturality of $\gamma$. A rigorous proof is omitted but can be assembled from results in \cite{Lefevre}.

\begin{prop} The category of filtered $\moduledata = (A_1 , \ldots, A_r | B_1, \ldots, B_s )$ polymodules is quasi-equivalent to the category of filtered left $A_1 \otimes \cdots \otimes A_r \otimes B_1^{op} \otimes \cdots \otimes B_s^{op}$-modules.
\end{prop}

From this, or from a direct argument, one obtains the following corollary which will be applied often implicitly.

\begin{cor} The category of filtered $(A_1, \ldots, A_r | B_1 , \ldots , B_s )$ polymodules is naturally equivalent to the category of filtered $(A_1, \ldots, A_r, B_1^{op}, \ldots, B_s^{op} | \K )$ polymodules. 
\end{cor}

Following \cite{BondalKapranov}, we observe that $\cmd$ is a pretriangulated category with sums and shifts defined in the obvious way and the natural cone construction $cone (\phi )$ given in the usual way. Namely,  $cone (\phi )$ is the graded vector space $ P \oplus \susp P^\prime$ and its structure morphism is
\begin{equation*} \mu_{cone (\phi )} = \left[ \begin{matrix}
\mu_P & \sigma \circ \phi \\ 0 & \mu_{\susp P^\prime} 
\end{matrix} \right] \end{equation*}
Given $\moduledata$ we let $\Unit_\moduledata = A_1 \otimes \cdots \otimes A_r \otimes B_1 \otimes \cdots \otimes B_s$ be the trivial polymodule whose structure map is induced by suspension, $\gamma$ and the algebra structure maps. A free polymodule is defined as a direct sum of copies of $\Unit_\moduledata$ and a projective polymodule as a direct summand of a free polymodule. A projective polymodule will be called finitely generated if it is a submodule of a finite sum of copies of $\Unit_\moduledata$. We define the subcategory of perfect $\moduledata$ polymodules to be the category $\pcmd$ of all polymodules quasi-isomorphic to a module built by finitely many cones of finitely generated projective polymodules. 

The concept of a polymodule is derived from the more natural notion of a differential comodule over several coalgebras in $Cog$. From this point of view, we have taken a backwards approach by defining the polymodule first, as the structure maps and definitions of morphisms are more transparent in the comodule setting. Nevertheless, we continue along our path full circle towards a realization of this structure as the bar construction of a polymodule.

Given an $\moduledata = (A_1, \ldots, A_r | B_1, \ldots, B_s)$ polymodule $(P, \mu_P)$, we take the free comodule $\cbr^\moduledata P$ as its bar construction (note that this is not free as a \dg comodule). We define its differential $b_P$ as
\begin{equation*} b_P = (\one_{\acbr A_1} \otimes \cdots \otimes \one_{\acbr A_r} \otimes \mu_P \otimes \one_{\acbr B_1} \otimes \cdots \otimes \one_{\acbr B_s} ) \circ \Delta_P + d^\prime_P . \end{equation*}
Then it follows from the defining equation \ref{eq:polymod} that $(\cbr^\moduledata P , b_P )$ is a left and right differential comodule over the coalgebras $\acbr A_i$ and $\acbr B_j$ respectively. We denote the \dg category of such \dg comodules with comodule morphisms as $\moduledata \mhyphen \ccm$ or simply $\ccm$. 

Given a morphism $\phi \in \Hom_{\moduledata \mhyphen \cmd} (P, P^\prime )$ we take $b_\phi : \cbr P \to \cbr P^\prime$ to be the map $b_\phi = (\one_{\acbr A_1} \otimes \cdots \otimes \one_{\acbr A_r} \otimes \phi_P \otimes \one_{\acbr B_1} \otimes \cdots \otimes \one_{\acbr B_s} ) \circ \Delta_P$. It then becomes an exercise that the bar construction gives a full and faithful functor from $\cmd$ to $\ccm$ whose essential image consists of free comodules. 

For our purposes, this is not enough as we wish to keep track of the length filtration throughout. The category $\moduledata \mhyphen \ccm$ has a natural embedding into $(\moduledata \mhyphen \ccm )^{lf}$ given by the primitive filtration. More concretely, given $(\mathbf{i}, \mathbf{j} ) = (i_1, \ldots, i_r, j_1, \ldots, j_s ) \in \Z^{r + s}$ we recall that $\cbr^\moduledata_{(\mathbf{i} , \mathbf{j} )} P $ is
\begin{equation*} \acbr_{i_1} A_1 \otimes \cdots \otimes \acbr_{i_r} A_r \otimes P \otimes \acbr_{j_1} B_1 \otimes \cdots \otimes \acbr_{j_s} B_s .\end{equation*}
This induces an embedding
\begin{equation} \cbr : \moduledata \mhyphen \cmd \to (\moduledata \mhyphen \ccm)^{lf}. \end{equation}
The induced length filtration on polymodule morphisms is then given by
\begin{equation} \mcf^{(\mathbf{i} , \mathbf{j} )} \Hom_{\cmd} (P, P^\prime ) = \{ \phi : \phi |_{\cbr_{-(\mathbf{i} , \mathbf{j} )} P } = 0 \} . \end{equation}
An advantage of the bar construction is the ease at which one sees the following proposition.
\begin{prop} The category $\moduledata \mhyphen \cmd$ is enriched over $\Z^{r + s}$-lattice filtered complexes. \end{prop}
In other words, morphism composition respects the total filtration on the tensor product. As stated above, this follows immediately from the definition of comodule morphism in the category $\ccm$. One should make certain not to confuse this enrichment with the notion that the objects of $\moduledata \mhyphen\cmd$ are lattice filtered, as this only occurs if we resolve the polymodules. 

\subsection{Filtered constructions}

In this section we define tensor products and inner homs of polymodules. To do this effectively, it is helpful to have a picture in mind as well as the appropriate notation associated to this picture.  We will say $\labs = (S^+ , S^- , \kappa)$ is a labelled set if $S^+$ and $S^-$ are finite sets and $\kappa$ is a function from $S^+ \sqcup S^-$ to the objects of $Alg_\infty$. We will write $A \in \labs$ (or $A \in \labs^\pm$) if there is $s \in S^+ \sqcup S^-$ (or $s \in S^\pm$) such that $\kappa (s) = A$. Given a labelled set $\labs = (S^+ , S^- , \kappa )$, we write $\labs^*$ for the labelled set $(S^- , S^+, \kappa)$. We take $\mcl$ to be the category of labelled sets with morphisms that are injective maps respecting the labelling. Note that $\mcl$ is closed under finite direct limits.

Given a labelled set $\labs = (\{t_1^+, \ldots, t_r^+ \}, \{t_1^-, \ldots, t_s^-\} , \kappa)$ we take $\labs \mhyphen \cmd$ to denote the category of $(\kappa (t_1^+), \ldots, \kappa (t_r^+)| \kappa (t_1^-)  , \ldots, \kappa ( t_s^-))$ polymodules. We abbreviate the differential coalgebra
\begin{equation*} \acbr [\kappa (t_1^+ )] \otimes \cdots \otimes \acbr [\kappa (t_r^+ )] \otimes \acbr [\kappa (t_1^- )] \otimes \cdots \otimes \acbr [\kappa (t_s^- )] \end{equation*}
by $\acbr A_\labs$. Any morphism $i: \labs_1 \to \labs_2$ induces a forgetful functor $i^* : \labs_2\mhyphen \cmd \to \labs_1\mhyphen \cmd$.

By gluing data $\glue = (\labs_0, \labs_1, \labs_2, i_1, i_2, j_1, j_2)$, we mean a pushout diagram as below in $\mcl$
\begin{figure}
\begin{tikzcd} \labs_0 \arrow{r}{i_1} \arrow{d}{i_2} & \labs_1^* \arrow{d}{j_1} \\ \labs_2 \arrow{r}{j_2} & \labs_1^* \sqcup_{\labs_0} \labs_2 \end{tikzcd}
\end{figure}
and we abbreviate $\labs_1 \sharp_{\labs_0} \labs_2$ for the labelled set $[\labs_1 - i_1 (\labs_0^*)]\sqcup [ \labs_2 - i_2(\labs_0)]$.

Given gluing data $\glue = (\labs_0, \labs_1, \labs_2, i_1, i_2, j_1, j_2)$, we define the tensor product as a functor
 \begin{equation*} \_ \ts{\labs_0} \_ : \labs_1 \mhyphen \cmd \times \labs_2 \mhyphen \cmd \to (\labs_1 \sharp_{\labs_0} \labs_2 \mhyphen \cmd )^{lf}. \end{equation*} 
As usual, this product is given by first passing through the bar construction,  applying the cotensor product and then recognizing the result as the bar construction of a polymodule. The details of this are now given.

Let $P_1, P_2$ be $\labs_1, \labs_2$ polymodules respectively. Then we let
\begin{equation*} P_1 \ts{\labs_0} P_2 = P_1 \otimes \acbr A_{\labs_0} \otimes P_2 . \end{equation*}
To simplify the definition of the structure map, we write $\Delta_1 = \Delta_{i_1^* (P_1)}$ and $\Delta_2 = \Delta_{i_2^* (P_2)}$ as partial comultiplications. These are the comultiplications obtained when considering $\cbr P_1$ and $\cbr P_2$ as comodules over $\acbr A_{\labs_0}$.  Then we see that there is an isomorphism of graded vector spaces:
\begin{equation*} \alpha : \cbr^{\labs_1 \sharp_{\labs_0} \labs_2}  (P_1 \ts{\labs_0} P_2 ) \to \cbr^{\labs_1} P_1 \square_{\acbr A_{\labs_0}} \cbr^{\labs_2} P_2 \end{equation*}
where $\square_{\acbr A_{\labs_0}}$ is the cotensor product (see, e.g. \cite{EM} ). Recall that this is the kernel of
\begin{equation*} \Delta_1 \otimes 1 - 1 \otimes \Delta_2 : \cbr^{\labs_1} P_1 \otimes \cbr^{\labs_2} P_2 \to \cbr^{\labs_1} P_1 \otimes  \acbr A_{\labs_0} \otimes \cbr^{\labs_2} P_2 . \end{equation*}
Restricting $\alpha$ to $P_1 \ts{\labs_0} P_2$, it is defined as $\alpha (p_1 \otimes a \otimes p_2) = p_1 \otimes \Delta_{\acbr A_{\labs_0}} (a) \otimes p_2$ where, as always, we implicitly use the symmetric monoidal map $\gamma$. It is extended to the bar construction by tensoring with the remaining coalgebras. Utilizing $\alpha$, one pulls back the differential from the cotensor product to obtain a differential $d$ on $\cbr^{\labs_1 \sharp_{\labs_0} \labs_2} (P_1 \ts{\labs_0} P_2)$. As this differential is a square zero comodule coderivation, it is induced by its composition with the projection 
\begin{equation*} \pi: \cbr^{\labs_1 \sharp_{\labs_0} \labs_2} (P_1 \ts{\labs_0} P_2) \to P_1 \ts{\labs_0} P_2 \end{equation*}
and one obtains the $A_\infty$-module map $\mu_{P_1 \ts{\labs_0} P_2} = \pi \circ d$.

Given morphisms $\phi_i : P_i \to P_i^\prime$ in $\labs_i$, we have that the cotensor product of the bar constructions 
\begin{equation*} b_{\phi_1} \square_{\acbr A_{\labs_0} } b_{\phi_2} : \cbr^{\labs_1 \sharp_{\labs_0} \labs_2} (P_1 \ts{\labs_0} P_2 ) \to \cbr^{\labs_1 \sharp_{\labs_0} \labs_2} (P_1^\prime \ts{\labs_0} P_2^\prime ) \end{equation*}
yields a natural map $\phi_1 \ts{\labs_0} \phi_2$ in $\labs_1 \sharp_{\labs_0} \labs_2 \mhyphen \cmd$. When considering $P_1 \ts{\labs_0} \_$ as a functor, we take $\phi_2$ to $1_{P_1} \ts{\labs_0} \phi_2$. Note that it follows from the definitions above and that of the cone that $P_1 \ts{\labs_0} \_$ is an exact functor. 

 Since the coalgebra $\acbr A_{\labs_0}$ is $\Z^{|\labs_0|}$-filtered by the primitives of $\acbr A$ for $A \in \labs_0$, we have that $P_1 \ts{\labs_0} P_2$ is lattice filtered by $\Z^{|\labs_0|}$. We will preserve this filtration in the definition and write
\begin{equation*} P_1 \tns{\gamma}_{\labs_0} P_2 = P_1 \otimes \left( \otimes_{A \in \labs_0^+} \acbr_{k_i}  A^{op} \right) \otimes \left( \otimes_{B \in \labs_0^-} \acbr_{l_j} B \right) \otimes P_2
\end{equation*}
where $\gamma = (k_1, \ldots, k_a, l_1, \ldots, l_b) \in \Z^{|\labs_0|}$.  Thus, we have obtained the above mentioned \dg functor.

It will be useful to have notation for filtered quotients in this setting. For this, we write
\begin{equation*} P_1 \odot^{\gamma}_{\labs_0} P_2 := \frac{P_1 \ts{\labs_0} P_2}{P_1 \tns{\gamma}_{\labs_0} P_2}. \end{equation*}

As expected, the tensor product of a given polymodule with the diagonal polymodule  yields a quasi-equivalent polymodule. However, the filtration is added structure which will be exploited later in the paper. For now, we simply define the natural quasi-equivalence and its inverse. Fix a labelled set $\labs = (S^+ , S^- , \kappa )$, let $2\labs = \labs^* \sqcup \labs$ and $\glue = (\labs, 2 \labs, \labs, i_1, i_2, j_1, j_2)$ the natural gluing data. We take $\Diagonal_\labs$ to be the diagonal $2\labs$ polymodule
\begin{equation*} \otimes_{t \in S^+ \cup S^-} \kappa (t) . \end{equation*}
The structure maps for $\Diagonal_\labs$ are simply the tensor products of the $A_\infty$ algebra maps composed with the shift for the various labelling algebras. Then we define the natural equivalences
\begin{equation} \label{eq:units} \xi_P : \Diagonal_\labs \ts{\labs} P \to P \hspace{.3in} \epsilon_P : P \to \Diagonal_\labs \ts{\labs} P .\end{equation}
Here $\xi_P$ is defined as the map induced by tensor multiplication $\mult $, the shift $\sigma$ and the polymodule multiplication map $\mu_P$,
\begin{equation*} \xi_P = \mu_P \circ (\mult \otimes 1_P) \circ (1_{\acbr A_{\labs}} \otimes \sigma \otimes 1_{\acbr A_{\labs}} \otimes 1_P). \end{equation*}
Letting $\unit_{\acbr A_{\labs}} : \K \to \acbr A_{\labs}$ send $1$ to $e_{\labs} = \otimes_{t \in S^+ \cup S^-} e_{\acbr [\kappa (s)]}$, we take
\begin{equation*}\epsilon_P = \sigma_{\otimes}^{-1} (e_\labs ) \otimes 1_{\acbr A_{\labs}} \otimes 1_P. \end{equation*}
Using the unital conditions, it is easy to verify that $\xi_P$ and $\epsilon_P$ are quasi-inverse maps. As a consequence, we obtain the following basic lemma which instructive as to the bar construction of a module.

\begin{lem} Suppose $A$ is an $A_\infty$ algebra and denote $A$ regarded as a right module over itself as $A^r$. Let $P$ be a left $A$ module, then the vector space $A^r \ts{} P$ is naturally quasi-isomorphic to $H^* (P )$. \end{lem}
\begin{proof} Let $\labs = (S^+ , S^- , \kappa)$ be the labelled set with $S^+ = \emptyset$, $S^- = \{t\}$ and $\kappa (t) = A$. Then the lemma follows from the fact that $A^r \ts{} P$ equals $\Diagonal_{2\labs} \ts{\labs} P$ as a complex. The latter is quasi-isomorphic to $P$ which has minimal model $H^* (P)$. \end{proof}

Combining this lemma with earlier remarks, we obtain the following important fact.

\begin{prop} Let $\glue = (\labs_0, \labs_1, \labs_2, i_1, i_2, j_1, j_2)$ be gluing data such that the algebras labelled by $\labs_0$ are compact. If $P_i$ are perfect $\labs_i$ polymodules then $P_1 \ts{\labs_0} P_2$ is a perfect $\labs_1 \sharp_{\labs_0} \labs_2$ polymodule.
\end{prop}

\begin{proof} The previous lemma implies that there is a quasi-isomorphism
\begin{equation*} \phi: \Unit_{\labs_1} \ts{\labs_0} \Unit_{\labs_2} \stackrel{q.i.}{\longrightarrow} \Unit_{\labs_1 \sharp_{\labs_0} \labs_2} \otimes \left( \oplus_{A \in \labs_0} H^* (A )  \right) . \end{equation*}
 By the compactness assumption, this implies that tensor products of finitely generated projective polymodules are finitely generated projective polymodules. Together with the definition of perfect modules and the fact that tensor product $\_ \ts{\labs_0} \_$ is exact, we have the result.
\end{proof}

To define the internal Hom, we again follow the approach for the tensor product and pass to coalgebras and comodules. There is an additional notion needed here from classical homotopy theory, that of a twisting cochain which we recall here. If $C$ is a \dg coalgebra and $A$ a \dg algebra, a map $\rho : C \to A$ is called a twisting cochain if $\partial \rho + \rho \cdot \rho = 0$ where $\partial \rho = d_A \rho - (-1)^{|\rho |} \rho d_C$ and $\rho \cdot \rho := \mult \circ \rho \otimes \rho \circ \Delta_C$ where $\mult$ is multiplication in $A$.

One of the central features of twisted cochains is that they allow one to define twisted tensor products \cite{Brown, Lefevre}. We take a moment to recall this construction for the case of a left module.

\begin{defn} Given a dg coalgebra $C$, a dg algebra $A$, a dg $C$ bicomodule $M$, a left dg $A$ module $N$ and a twisting cochain $\rho: C \to A$, the twisted tensor product $M \otimes_\rho N$ (or $N \otimes_\rho M$) is defined as the ordinary tensor product of vector spaces with chain map $d_M \otimes 1_N + 1_M \otimes d_N + \rho \cap \_$ where 
\begin{equation*} \rho \cap \_ = (1_M \otimes \mult_N ) \circ (1_M \otimes \rho \otimes 1_N) \circ (\Delta_M \otimes 1_N ). \end{equation*}
The result is a left (or right) $C$ comodule.
\end{defn}

The case of right module and bimodule is analogous.

Now, let $ \labs = \labs^\prime \sqcup \labs^{\prime \prime}$ in $\mcl$ and $i: \labs^\prime \to \labs$, $j: \labs^{\prime \prime} \to \labs$   the inclusion maps. Given a $\labs$ polymodule $P$, we define a map
\begin{equation*} \rho_{j} : \acbr A_{\labs^\prime} \to \Hom_{\labs^{\prime \prime} \mhyphen\cmd} (j^*(P), j^*(P)) \end{equation*}
as $[\rho_{j} (\mathbf{c})] (\mathbf{a} \otimes p \otimes \mathbf{b} ) = \mu_P (\mathbf{c} \otimes \mathbf{a} \otimes p \otimes \mathbf{b} )$ where $\mathbf{a} \otimes p \otimes \mathbf{b} \in \cbr^{\labs^{\prime \prime}} P$. It follows from equations \ref{eq:polymod} and \ref{eq:polymor} that $\rho_{j}$ is a twisting cochain from the \dg coalgebra $ \acbr A_{\labs^\prime}$ to the \dg algebra $\Hom_{\labs^{\prime \prime} \mhyphen \cmd} (j^*(P), j^*(P))$.

Suppose $\glue = (\labs_0, \labs_1, \labs_2, i_1, i_2, j_1, j_2) $ is gluing data and $P_1, P_2$ are $\labs_1^*, \labs_2$ polymodules respectively. Then, as a graded vector space, we define $\inthom_{\labs_0} (P_1, P_2 )$ as $\Hom_{\labs_0 \mhyphen\cmd } (i_1^*(P_1) , i_2^* (P_2) )$. The structure map 
\begin{equation*}\mu_{\inthom_{\labs_0} (P_1, P_2 )} : \cbr^{\labs_1^* \sharp_{\labs_0} \labs_2} \inthom_{\labs_0} (P_1, P_2 ) \to \inthom_{\labs_0} (P_1, P_2 )\end{equation*}
is set to equal the differential on the twisted tensor product composed with the projection $\pi : \cbr^{\labs_1^* \sharp_{\labs_0} \labs_2} \inthom_{\labs_0} (P_1, P_2 ) \to \inthom_{\labs_0} (P_1 , P_2 )$, where the former is induced by the isomorphism
\begin{equation*} \cbr^{\labs_1^* \sharp_{\labs_0} \labs_2} \inthom_{\labs_0} (P_1, P_2 ) = \acbr A_{\labs_2 - i_2 (\labs_0)}  \otimes_{\rho_{i_2}} \inthom_{\labs_0} (P_1 , P_2 ) \otimes_{\rho_{i_1}} \acbr A_{\labs_1 - i_1 (\labs_0)}  \end{equation*}
Again we keep track of the lattice filtration so that $\inthom_{\labs_0} (P_1 , P_2 )$ is a $\Z^{|\labs_0 |}$ filtered polymodule.

As in the case of the tensor product, for any $\labs$ polymodule $P$, the diagonal polymodule $\Diagonal_\labs$ plays the role of a unit for $\inthom_\glue (\Diagonal_\labs , P )$. Again we define the natural transformations
\begin{equation*}  \chi_P : \inthom_\labs ( \Diagonal_\labs , P ) \to P \hspace{.3in} \upsilon_P : P \to \inthom_\labs ( \Diagonal_\labs , P ) . \end{equation*}
Where $\chi_P  (\mathbf{a} \otimes \phi \otimes \mathbf{b}) = (-1)^{|\phi| |\mathbf{a}|} \phi (\mathbf{a} \otimes \unit_{\acbr A_{\labs} (1)} \otimes \mathbf{b}  )$ and $\upsilon_P$ is the strict map sending $p$ to the morphism $\phi_p$ defined as $\phi_p (\mathbf{a} \otimes \mathbf{q} \otimes \mathbf{b} ) = \mu_P (\mathbf{a} \otimes \mathbf{q}^+ \otimes p \otimes \mathbf{q}^- \otimes \mathbf{b})$ where $\mathbf{q}^\pm$ is the tensor factor of $q$ in $\Diagonal_{\labs^\pm}$.

\subsection{Filtered adjunction} 

In this section we observe the classic adjunction between tensor product and internal Hom for polymodules. This leads to elementary, but powerful, observations on dual $A_\infty$-modules. We will be concerned with preserving the lattice filtrations naturally throughout. 

To state the theorem, we need to specify the gluing data between three categories of polymodules. Assume $\labs_i$ are labelled sets for $i = 1, 2, 3$. We say that the data $\gluecyc = (\glue_{12}, \glue_{23}, \glue_{31})$ form a gluing cycle if $\glue_{ij}$ are the gluing data
\begin{eqnarray*} \glue_{12} & = & (\labs_{12}, \labs_1, \labs_2 , i_{12}, i^\prime_{12}, j_{12}, j^\prime_{12} ), \\
\glue_{23} & = & (\labs_{23}, \labs_2^*, \labs_3 , i_{23}, i^\prime_{23}, j_{23}, j^\prime_{23} ) ,\\
\glue_{31} & = & (\labs_{31}, \labs_3, \labs_1^* , i_{31}, i^\prime_{31}, j_{31}, j^\prime_{31} ), \end{eqnarray*}
and $image (i_{kl} )$ is disjoint from $image (i_{mk}^\prime )$. A gluing cycle can be represented graphically as a directed graph with three vertices. Vertices $v_1, v_2$ have incoming and outgoing edges $\labs_i^\mp$ and $v_3$ has incoming and outgoing edges $\labs_3^\pm$. Those edges that connect vertices $v_i$ and $v_j$ form the labelled set $\labs_{ij}$. This is depicted in Figure \ref{fig:gluecycle} below.
%
\begin{figure}[h]
\begin{picture}(0,0)%
\includegraphics{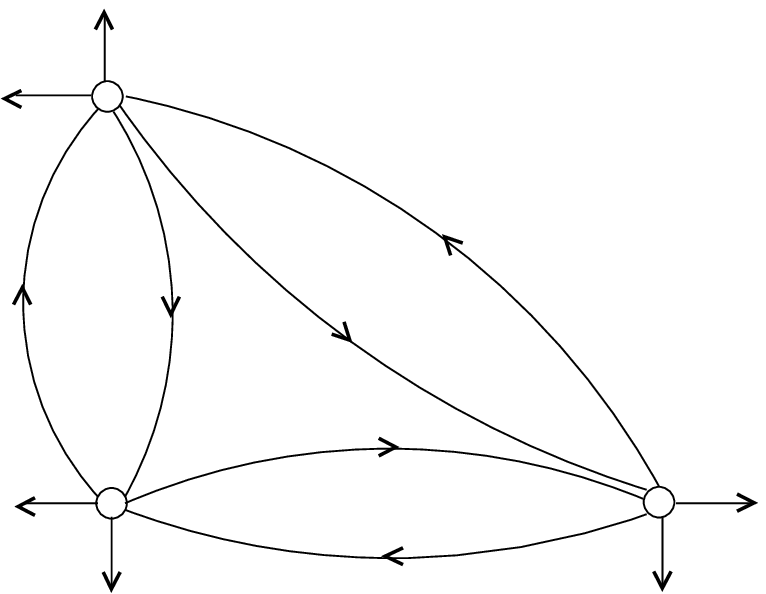}%
\end{picture}%
\setlength{\unitlength}{4144sp}%
\begingroup\makeatletter\ifx\SetFigFont\undefined%
\gdef\SetFigFont#1#2#3#4#5{%
  \reset@font\fontsize{#1}{#2pt}%
  \fontfamily{#3}\fontseries{#4}\fontshape{#5}%
  \selectfont}%
\fi\endgroup%
\begin{picture}(3473,2697)(2988,-3355)
\put(4719,-3004){\makebox(0,0)[lb]{\smash{{\SetFigFont{8}{9.6}{\rmdefault}{\mddefault}{\updefault}{ $\vdots$}%
}}}}
\put(3095,-952){\makebox(0,0)[lb]{\smash{{\SetFigFont{8}{9.6}{\rmdefault}{\mddefault}{\updefault}{ $\labs_1^*$}%
}}}}
\put(6081,-3191){\makebox(0,0)[lb]{\smash{{\SetFigFont{8}{9.6}{\rmdefault}{\mddefault}{\updefault}{ $\labs_2^*$}%
}}}}
\put(4864,-2178){\makebox(0,0)[lb]{\smash{{\SetFigFont{8}{9.6}{\rmdefault}{\mddefault}{\updefault}{ $\labs_{12}$}%
}}}}
\put(3126,-3191){\makebox(0,0)[lb]{\smash{{\SetFigFont{8}{9.6}{\rmdefault}{\mddefault}{\updefault}{ $\labs_3$}%
}}}}
\put(3289,-2040){\makebox(0,0)[lb]{\smash{{\SetFigFont{8}{9.6}{\rmdefault}{\mddefault}{\updefault}{ $\cdots$}%
}}}}
\put(3289,-2227){\makebox(0,0)[lb]{\smash{{\SetFigFont{8}{9.6}{\rmdefault}{\mddefault}{\updefault}{ $\labs_{31}$}%
}}}}
\put(4707,-1989){\makebox(0,0)[lb]{\smash{{\SetFigFont{8}{9.6}{\rmdefault}{\mddefault}{\updefault}{ $\cdots$}%
}}}}
\put(4478,-2974){\makebox(0,0)[lb]{\smash{{\SetFigFont{8}{9.6}{\rmdefault}{\mddefault}{\updefault}{ $\labs_{23}$}%
}}}}
\end{picture}%
\caption{\label{fig:gluecycle} The gluing cycle $\gluecyc$}
\end{figure}

We take $\labs_\gluecyc$ to be the labelled set $(\labs_1 \sharp_{\labs_{12}} \labs_2)^* \sharp_{\labs_{23} \sqcup \labs{31}} \labs_3$, i.e. $\labs_\gluecyc$ consists of the half edges in figure \ref{fig:gluecycle}. With this notation, we can prove the following classic adjunction:
\begin{thm} Given a gluing cycle $\gluecyc$ and polymodules $P_i \in \labs_i \mhyphen \cmd$, there is a natural isomorphism $\Phi$ in $(\labs_\gluecyc \mhyphen \cmd )^{lf} $,
\begin{equation}\label{eq:adj2} \Phi : \inthom_{\labs_{31} \sqcup \labs_{23}} (P_1 \ts{\labs_{12}} P_2 , P_3 ) \to \inthom_{\labs_{31} \sqcup \labs_{12}} (P_1 , \inthom_{\labs_{23}} (P_2 , P_3)). \end{equation}
\end{thm}

\begin{proof} This is simply an exercise in the definitions of the last section and the observation that $\gr^{lf}$ is a closed category. Letting $\star =  \inthom_{\labs_{31} \sqcup \labs_{23}} (P_1 \ts{\labs_{12}} P_2 , P_3 ) $ we have the following natural isomorphisms of $\Z^{|\labs_{31}| + |\labs_{12} | + |\labs_{23}|}$ filtered graded vector spaces

\begin{eqnarray*}\star & = & \Hom_{(\labs_{31} \sqcup \labs_{23}) \mhyphen \cmd} (P_1 \ts{\labs_{12}} P_2 , P_3 ) \\ & = &  \Hom_\gr (\acbr A_{\labs_{31}}  \otimes 
P_1 \otimes \acbr A_{\labs_{12}} \otimes P_2 \otimes \acbr A_{\labs_{23}} , P_3 ) \\ & \simeq & \Hom_\gr ( \acbr A_{\labs_{31}} \otimes P_1 \otimes \acbr A_{\labs_{12}}  , \Hom_\gr ( P_2 \otimes \acbr A_{\labs_{23}} , P_3) ) \\ & = & \inthom_{\labs_{31} \sqcup \labs_{12}} (P_1 , \inthom_{\labs_{23}} (P_2 , P_3)) \end{eqnarray*}

To complete the proof, one must show that the isomorphisms above respect the differentials, which follows immediately from the definitions.
\end{proof}

The same proof gives a natural equivalence

\begin{equation*}  \Phi^l : \inthom_{\labs_{31} \sqcup \labs_{23}} (P_1 \ts{\labs_{12}} P_2 , P_3 ) \to \inthom_{\labs_{32} \sqcup \labs_{23}} (P_2 , \inthom_{\labs_{13}} (P_1 , P_3)),\end{equation*}
making the bicategory of $A_\infty$-algebras and bimodules into a biclosed bicategory. We apply this theorem to a simple gluing cycle to obtain the following corollary.
\begin{cor} Suppose $A \in Alg_\infty$ and $P \in A \mhyphen \cmd$. Then:
\begin{equation} \ell^{\_ \ts{A} P} \leq \ell^{\inthom_{A} (P , \_ )} \end{equation} \end{cor}
\begin{proof} Here we take $\labs_1 = \labs_2^*$ to be the labelled set $S^- = \{A \}$ and $S^+ = \emptyset$ while $\labs_3$ is just the empty labelled set. We take $P_1 = Q$ to be any $A$ module and $P_2 = P$, $P_3 = \K$. Then the filtered adjunction \ref{eq:adj2} reads $\inthom_\K (Q \ts{A} P , \K ) \simeq \inthom_A (P, \inthom_\K (Q, \K ))$. By universal coefficients, the left hand side has length equal to $\ell (Q \ts{A} P )$ while the right hand side has length $ \ell ( \inthom_A (P, \inthom_\K (Q , \K )))$. As the $Q$ is arbitrary, we have then that the supremum $\ell^{\_ \ts{A} P}$ is less than or equal to the supremum $\ell^{\inthom_A (P ,  \_ )}$ verifying the claim. 
\end{proof}
For formal algebras concentrated in degree zero, the above corollary is the elementary fact that flat dimension is less than or equal to projective dimension. We note that for arbitrary (formal and non-formal) algebras $A$, it is not the case that all left modules are quasi-isomorphic to $\inthom_\K (Q, \K )$ for some $Q$, so just as in the formal setting, this inequality can be strict. We will observe conditions for which this inequality is an equality below.

The dual $P^\vee$ of an $\labs$ polymodule $P$ is the $\labs^*$ polymodule $\inthom_\K (P, \K )$. We start with an elementary lemma for perfect polymodules over compact algebras.

\begin{prop} Suppose $\labs$ labels compact algebras. Then 
\begin{equation*}\_^\vee : \labs -\pcmd \to \labs^* -\pcmd \end{equation*} is an equivalence of categories and there is a natural isomorphism $\Theta : I \to (I^\vee )^\vee$. 
\end{prop}

\begin{proof} We prove this for the case of $\labs $ labelling a single  compact algebra $A$ as the general case is the same.
Every perfect $A$ module $P$ has a finite dimensional minimal model $P_{min}$ defined uniquely up to isomorphism. Thus there is the usual graded vector space natural isomorphism $\Theta_{\gr} : P_{min} \to (P_{min}^\vee )^\vee$ defined in the usual way $[\Theta_{\gr} (p) ] (l) = (-1)^{|l| |p|} l (p)$. It is immediate from the definition of internal hom that $\Theta :=  \Theta_{\gr}$ is indeed a strict $A$-module homomorphism. 
\end{proof}

To generalize this proposition, we fix gluing data $\glue$ between $\labs_1^*$ and $\labs_2$. The following proposition, which was observed early in homological algebra, is stated below in terms of polymodules.

\begin{prop} Suppose $P_i$ is a $\labs_i$ polymodule and $P_2$ is a perfect $\labs_2 $ polymodule. If the algebras labelled by $\labs_2 - i_2 (\labs_0)$ are compact, then there is a natural filtered quasi-equivalence
\begin{equation*} P_1^\vee \ts{\labs_0} P_2 \simeq \inthom_{\labs_0} (P_2, P_1 )^\vee. \end{equation*}
\end{prop}

\begin{proof} First we define a morphism $\Psi : P_1^\vee \ts{\labs_0} P_2 \to \inthom_{\labs_0} (P_2 , P_1 )^\vee $  of filtered $\labs_1^* \sharp_{\labs_0} \labs_2$ polymodules by
\begin{equation*} [\Psi ( \mathbf{a} \otimes \phi \otimes \mathbf{b} \otimes p \otimes \mathbf{c} )] (\psi) = (-1)^{|\psi| (|p| + |\mathbf{b}| + |\mathbf{c}|) + |\phi| |\mathbf{a}|} \phi (\mu_{\inthom_{\labs_0} (P_2 , P_1 )} (\mathbf{a} \otimes \psi \otimes \mathbf{c}) (\mathbf{b} \otimes p )).
\end{equation*}
It is plain to see that $\Psi$ preserves the lattice filtrations and that $\Psi$ is a natural transformation. 

Now we check to see that $\Psi$ is a quasi-isomorphism for $P_2 = \Unit_{\labs_2}$. Write $\labs_2^\prime $ for $\labs_2 - i_2 (\labs_0 )$ and note that $\Unit_{\labs_2} =  \Unit_{\labs_0} \otimes \Unit_{\labs_2^\prime}$. By choosing minimal models for the algebras labelled by $\labs_2^\prime$ we may assume $\Unit_{\labs_2^\prime}$ is a finite dimensional vector space over $\K$. This gives 
\begin{equation*} P_1^\vee \ts{\labs_0} \Unit_{\labs_2} = ( P_1^\vee \ts{\labs_0} \Unit_{\labs_0} ) \boxtimes \Unit_{\labs_2^\prime}. \end{equation*}
While on the other side we obtain
\begin{eqnarray*} \inthom_{\labs_0} (\Unit_{\labs_0} , P_1 )^\vee & = & \inthom_{\labs_0} (\Unit_{\labs_0} , P_1 )^\vee \boxtimes (\Unit_{\labs_2^\prime}^\vee )^\vee ,\\ & = & \inthom_{\labs_0} (\Unit_{\labs_0} , P_1 )^\vee \boxtimes \Unit_{\labs_2^\prime}, \end{eqnarray*}
where the last equality follows from the compactness assumption. It is easy to see that $\Psi$ factors through this tensor decomposition of $\Unit_{\labs_2}$, so we need only show the equivalence on the $\Unit_{\labs_0}$ factor.

For this, observe that the tensor product and internal Hom with $P_2 = \Unit_{\labs_0}$ yields the same complex as $P_2 = \Diagonal_{\labs_0}$ we restrict $\xi$ to obtain the following quasi-commutative diagram
\begin{figure}
 \begin{tikzcd}P_1^\vee \ts{\labs_0} \Unit_{\labs_0} \arrow{rr}{\Phi} \arrow{rd}[swap]{\xi_{P_1^\vee}} & & \inthom_{\labs_0} ( \Unit_{\labs_0}, P_1 )^\vee  \\ & P_1^\vee \arrow{ur}[swap]{(\chi_{P_1} )^\vee} & \end{tikzcd}
\end{figure}



By exactness of $ P_1^\vee \ts{\labs_0} \_$ and $\inthom_{\labs_0} (\_ , P_1 )^\vee$ and naturality of $\Psi$, we have that $\Psi$ induces a quasi-isomorphism on perfect $\labs_2$ polymodules. As was observed above, $\Psi$ respects filtrations which yields the claim.
\end{proof}

As a corollary, we have the following important fact

\begin{cor}\label{cor:prequalsfl} Suppose $A \in Alg_\infty$ and $P \in A - \pcmd$. Then
\begin{equation*} \ell^{\Hom (P , \_ )} = \ell^{\_ \ts{} P } . \end{equation*}
\end{cor}

This equality motivates the following definition.

\begin{defn} For $P \in A \mhyphen \pcmd$ we define the length of $P$ to be 
 \begin{equation*} \ell (P) :=  \ell^{\Hom (P , \_ )} = \ell^{\_ \ts{} P }, \end{equation*}
 and define the global length of $A \mhyphen \pcmd$ to be the supremum 
 \begin{equation*} \ell (A \mhyphen \pcmd) :=  \sup \{\ell (P ): P \in A \mhyphen \pcmd \}. \end{equation*}
\end{defn}
\section{Dimensions of $A_\infty$ categories}

In this section we lift many of the definitions and theorems of the dimension theory for triangulated categories to the pretriangulated setting. After recalling some definitions and results from triangulated and pretriangulated categories from  \cite{BFK, Lefevre, Orlov, Rouquier}, we prove our first main theorem that equates filtered length of internal homs with the generation time of a given object. We follow this with a proof of the base change formula for $A_\infty$-algebras. 

\subsection{Generators in triangulated categories}

We take a moment to recall some definitions and notation from
\cite{Rouquier}. Given a triangulated category $\mcc$ and a
subcategory $\mci$, we define $\cl{I}$ to be the smallest full
subcategory of $\mcc$ closed under direct summands, finite direct
sums and shifts. Given two subcategories $\mci_1 , \mci_2 \subset
\mct$, we define $\mci_1 * \mci_2$ to be the category of objects
$N$ such that there exists a distinguished triangle
\begin{equation*} M_1 \to N \to M_2 \to \end{equation*} 
in $\mct$ with $M_1 \in \mci_1$ and $M_2 \in \mci_2$. We take $\mci_1
\diamond \mci_2 := \cl{\mci_1 * \mci_2}$. It follows from the
octahedral axiom that $\diamond$ is an associative operation, so
the category $\mci^{\diamond d}$ is well defined. With this
notation in hand, the following definitions can be stated.

\begin{defn} Let $\mct \subseteq \mcc$. 
\begin{itemize}
 \item[i)] $\mci$ generates $\mct$ if given $N \in
\mct$ with $\Hom_\mct ( M[i], N ) = 0$ for all $M \in \mci$ and all
$i \in \Z$, then $N = 0$.

\item[ii)] $\mci$ is a $d$-step generator of $\mct$ if $\mct =
\mci^{\diamond d}$.

\item[iii)] $\mct$ is finitely generated if there exists $G \in \mct$
which generates $\mct$. In this case we call $G$ a generator for
$\mct$.

\item[iv)] $\mct$ is strongly finitely generated if there exists $M \in
\mct$ which is a $d$-step generator.
\end{itemize}
\end{defn}

We utilize the above definitions to define level and dimension as
follows.

\begin{defn} If $G$ generates $\mct$ and $M \in \mct$ we say the
level of $M$ with respect to $G$ is
\begin{equation*} \lvl_G (M) = \min \left\{ d : M \in \cl{G^{\diamond
(d - 1) } } \right\} \end{equation*} and the generation time of $G$ is
\begin{equation*} \gt (G) = \min \left\{ d : \mct =
\cl{G^{\diamond (d - 1) } } \right\}
\end{equation*}
The dimension of a category $\mct$ with generators is defined to
be the smallest generation time. The Orlov spectrum of $\mct$ is the set of all generation times.
\end{defn}


The central theme of this paper is to enhance the above
definitions into the language of \dg and $A_\infty$-categories.
Thus we will assume our category $\mct$ is always a subcategory of
the homotopy category $H^0 (\mca )$ for some pretriangulated
$A_\infty$-category $\mca$. If $\mct$ is an algebraic triangulated category, this is implied by a theorem of Lef\`evre-Hasegawa which we site below. First we fix notation and, in triangulated category $\mct$, write $\Hom_{\mct}^* (M, N)$ for the algebra $\oplus_{n \in \Z} Hom_{\mct} (M, N[n])$.

\begin{thm}[7.6.0.4, \cite{Lefevre}] \label{thm:LF} If $\mct$ is an algebraic triangulated which is strongly generated by an object $G$, then there is an $A_\infty$ structure on $A_G := \Hom_\mct^* (G, G)$ such that the Yoneda functor evaluated at $G$ from $\mct$ to $H^0 (A_G \mhyphen \pcmd)$ is a triangulated equivalence.
\end{thm}

We say that a pretriangulated $A_\infty$-subcategory $\mcb$ (strongly) generates if $H^0 (\mcb)$
does in $H^0 (\mca )$. We also use the same language and notation
as above for level, generation time and dimension. 


\begin{figure}
 \includegraphics{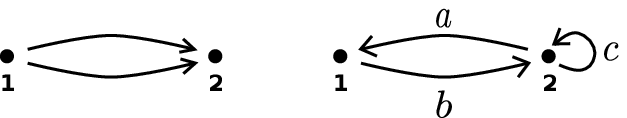}
 \caption{\label{fig:quiv1} Quivers of generators for $\mcd^b (\p^1 )$}
\end{figure}

Before proceeding with this discussion, we take a moment to illustrate this theorem with some examples.

\begin{eg} For $\p^n$, Beilinson showed (\cite{Beilinson}) that $\langle \mco, \mco (1) , \cdots \mco (n) \rangle$ forms a full exceptional collection for $\mcd^b (\p^n )$. Taking $G = \oplus_{i = 0}^n \mco (i)$ then gives a generator. From grading considerations, the endomorphism algebra $A_G$ has no higher products so $\mcd^b (\p^n ) \simeq H^0 (A_G \mhyphen \pcmd )$. In the case of $n = 1$, $A_G$ is the path algebra of the Kronecker quiver illustrated in Figure \ref{fig:quiv1}.
\end{eg}

Exceptional collections in the dimension theory of triangulated categories were studied in \cite{BFK}. In general, one can mutate an exceptional collection to obtain a new exceptional collection. Below we examine one such mutated case.

\begin{eg} Let $n = 1$, then mutating $\langle \mco , \mco(1) \rangle$ we obtain the collection $\langle \mco , \mco_p \rangle$. The algebra $A_{G^\prime}$ is the quiver algebra with relations given in the middle of Figure \ref{fig:quiv1} where $\deg (a) = 1 = \deg (c)$ and $\deg (b) = 0$ and $ba = c$. Here, the grading does not preclude the existence of higher products, but it is not hard to exhibit a quasi-isomorphism from this algebra to the \dg algebra endomorphism algebra of the mutated objects in $A_G \mhyphen \pcmd$. Again we obtain the isomorphism $\mcd^b (\p^1 ) \simeq H^0 (A_{G^\prime} \mhyphen \pcmd )$ from coherent sheaves to graded modules over the graded algebra $A_{G^\prime}$.
\end{eg}



\begin{eg} \label{eg:3} Another studied example is the category of matrix factorizations for the function $f_n : \C \to \C$ via $f_n (z) = z^n$, or equivalently the derived category of singularities $D^b_{sg} (f_n^{-1} (0))$. It was observed in \cite{BFK} that every non-zero object of $MF(\C[[z]], f_n)$ is a strong generator and that the generator
\begin{equation*}
\begin{tikzcd} \C 
\rar[-to, to path={([yshift=0.5ex]\tikztostart.east) -- ([yshift=0.5ex]\tikztotarget.west) \tikztonodes}]{z^{n - 1}}
\rar[to-, to path={([yshift=-0.5ex]\tikztostart.east) -- ([yshift=-0.5ex]\tikztotarget.west) \tikztonodes}][swap]{z} 
& \C \in MF (\C[[z]],f_n) \end{tikzcd}
\end{equation*}
or $\mco_0 \in D^b_{sg} (f_n^{-1} (0))$ had maximal generation time. Also, in \cite{Dyckerhoff}, the computation of a minimal model for $A_G$ as a $\Z / 2\Z$ graded $A_\infty$-algebra was performed and found to equal $A_G = k[\theta] / (\theta^2)$ where $\deg (\theta ) = 1$ and all higher products vanish except $\mu^n (\theta, \theta, \cdots, \theta) = 1$. Again, we have $D^b_{sg} (f_n^{-1} (0)) \simeq H^0 (A_G \mhyphen \pcmd)$.
\end{eg}



Now, starting with an $A_\infty$ pretriangulated category $\mca$, let $G \in \mca$ be a generator and $A_G = \Hom^* (G, G)$ its $A_\infty$ endomorphism algebra. We define the $A_\infty$-functor $\text{ev}_G :
\mca \to  A_G \mhyphen \cmd$ via
\begin{equation*} \text{ev}_G (B ) = \Hom^* (G, B) \end{equation*} The map
$T^1$ on morphisms is composition and $T^k$ is defined using
higher multiplication. Clearly, $\text{ev}_G$ factors through the Yoneda
embedding and can be thought of as evaluation of Yoneda at the
point $G$. From Theorem \ref{thm:LF}, if $G$ is a strong generator, we have that that $\text{ev}_G$ is an quasi-equivalence with $A_G \mhyphen \pcmd$. In
particular, given any two objects, $M, N \in \mca$, the associated
map
\begin{equation} \label{eq:mor} H^* \text{ev}_G: H^* (\Hom^* (M, N)) \to
H^* (\Hom^* (\text{ev}_G M, \text{ev}_G N))
\end{equation} is an isomorphism. In a moment, we will examine the right hand side of \ref{eq:mor}.

Were we to have started out in the triangulated setting, we could
have defined the functor $\textbf{ev}_G : H^* \mca \to H^*
(A_G ) \mhyphen \pmd$. It is well known that the natural functor $\Phi: H^*
({A_G}\mhyphen \pcmd ) \to {H^* (A_G)}\mhyphen \pmd$ is not an equivalence of
categories. However, all of these categories and functors fit into
the diagram of categories below.
\begin{figure}[h]
\begin{tikzcd}
 \mca \arrow[maps to]{r}{H^*} \arrow{dd}[swap]{\text{ev}_G}{\simeq} & H^* (\mca) \arrow{dd}{H^* \text{ev}_G}[swap]{\simeq} \arrow{rd}{\textbf{ev}_G} & \\  & & {H^* (A_G)}\mhyphen \pmd \\
  A_G \mhyphen \pcmd   \arrow[maps to]{r}{H^*} & H^* (A_G \mhyphen \pcmd) \arrow{ru}{\Phi} & 
\end{tikzcd}
\end{figure}

The kernel (i.e. all morphisms sent to zero) of $\text{ev}_G$ is
defined to be the G-ghost ideal. We write this ideal as $\mcg_G$
and its $n$-th power to be $\mcg_G^n$. The following lemma will be
important for what follows and can be found in \cite{BFK}.

\begin{lem}[The Ghost Lemma] If $\mct$ is an algebraic triangulated category with strong generator $G$ such that $A_G$ is compact. Then $M \in \cl{G^{\diamond (d ) } } $ and $M \not\in \cl{G^{\diamond (d - 1) } }$ if and only if there exists an $N \in \mct$ such that $\mcg_G^d \Hom^* (M, N) \ne 0$.
\end{lem}

An important conceptual point about this perspective is that, by
choosing a generating object $G$, we have enhanced the 
on the homotopy category of $\mca$ to a filtered category. This is not an invariant of
the $A_\infty$-category $\mca$, nor is it an invariant of the
triangulated category $H^* \mca$. It is additional structure
introduced by the choice of generator which provides homological
information relative to $G$.


\subsection{Ghosts and Length}

We will now establish the link between generation
time and filtration length. The following lemma is straightforward, but we supply a proof to establish some notation.

\begin{lem}\label{lem:basic} In $A\mhyphen \cmd$ we have $\ell (\Unit_A ) = 0$. \end{lem}

Before we begin the proof, we define a
weaker class of maps $ \text{Map}_{C}^{k} (P ,P^\prime)$ between two differential, free comodules of a coalgebra $C$. Given any map
of comodules $f: P \to P^\prime$, we take
\begin{equation*} [\Delta , f] := \Delta_{P^\prime} f -
(1_{C} \otimes f ) \Delta_P \end{equation*}
Note that
\begin{equation*} [\Delta , f g] = (1_{C^\prime} \otimes f \otimes 1_{C}) [\Delta
, g] + [\Delta , f] g
\end{equation*} For $k \geq 0$, define

\begin{equation*} Map^{k} (P, P^\prime ) = \left\{ f :  \text{image} ([\Delta , f])  \subset {C} \otimes
P^\prime_{[k]} \right\} \end{equation*}

These classes of maps will be useful when defining homotopies. Indeed, they naturally appear in the cobar
complex of morphisms from the cobar of $P$ to the cobar of
$P^\prime$ satisfying filtration properties on their differential in that
complex. A straightforward generalization of the above definition to bicomodules will also be used. We now record some basic properties.

\begin{lem}\label{lem:map1}  \item i) If $f \in Map^{k} (P, P^\prime )$, $g \in  \mcf^{i} \Hom (P^\prime ,
P^{\prime \prime})$ with $k \leq i$, then $g f \in
\mcf^{i - k} \Hom (P, P^{\prime \prime})$.

\item  ii) If $f \in Map^{k} (P, P^\prime )$
then $f (P_{[n]}) \subseteq P^\prime_{[n + k]}$.

\item iii) If $f \in Map^{k} (P, P^\prime )$, $g \in Map^{i} (P^\prime , P^{\prime \prime})$ then $g f \in Map^{k +i }
(P, P^{\prime \prime} )$.

\item iv) If $f \in Map^{k } (P, P^\prime ) $ then $\partial f
\in Map^{k} (P, P^\prime )$.

\end{lem}

We proceed with the proof of Lemma \ref{lem:basic}.

\begin{proof} In order to prove this lemma, we define a homotopy
contraction \begin{equation*} h_A : \br \Unit_A \to \br \Unit_A \end{equation*} as $\sum_{m = 1}^{\infty}
1^{\otimes m} \otimes \eta$ where $\eta$ is the insertion of the
identity. More concretely,

\begin{equation*} h_A ([a_1| \cdots |a_m]) = (-1)^{|a_1| + \cdots + |a_m|}[ a_1| \cdots |a_m|e]
\end{equation*} where $|a_i |$ is the degree of $a_i$ in $A[1]$. A quick
computation shows that indeed

\begin{equation*} h_A b_A + b_A h_A = 1  \end{equation*}

so that $h_A$ is a vector space contracting homotopy of $\cbr \Unit_A$.

Note that $h_A$ is not a $\acbr A$-comodule morphism of $\cbr \Unit_A$
(otherwise, the entire category $A\mhyphen \cmd$ would be zero). Indeed,
we have, for any $a \in \cbr \Unit_A$, \begin{equation*} (\Delta h_A - (1
\otimes h_A ) \Delta )(a) = (-1)^{|a |} a \otimes [e]
\end{equation*}

This implies that $h_A \in Map_{\acbr A}^{1} (\cbr \Unit_A, \cbr \Unit_A)$.
By \ref{lem:map1}, we have that if $\phi \in \mcf^1 \Hom_{\cmd} (\Unit_A,
M)$ then $b_\phi \circ h_A \in Map_{\acbr A}^0 (\cbr \Unit_A, \cbr \Unit_A)$ is
a comodule morphism. Thus, if $\phi \in \mcf^1 \Hom_{\cmd} (A, M)$
is a homomorphism, then $\partial (b_\phi h_A) = b_\phi
\partial h_A = b_\phi$ implying that it is a boundary and
therefore $F^1 \Hom (\Unit_A , M) = 0$.
\end{proof}

Applying this lemma yields the following corollary.

\begin{cor} For any $A_\infty$-algebra $A$, $\mcg_A = F^1$. \end{cor}

\begin{proof} Clearly, if $ \phi : M \to N$ is in $\mcf^1$, then $\phi_*
:\Hom_{\cmd} (A, M) \to \mcf^1 \Hom (A, N)$ so $[\phi ]_* = 0$.
Conversely, using the homotopy retract above, one sees that that
there is a map natural with respect to $K$
\begin{eqnarray*} \Hom_\cmd (A, K) & \to & \Hom_\cmd (A, K) / \mcf^1
\Hom_\cmd (A, K) \end{eqnarray*} which induces a natural inclusion
\begin{eqnarray*} \Hom (A, K)  & \hookrightarrow & \Hom_{mod} (H (A), H(K))
\\ & \simeq & H (K)
\end{eqnarray*} Thus if $[\phi]_* = 0$ then $[\phi^0] = 0$ implying
$[\phi] \in F^1 \Hom (M, N)$. \end{proof}

Owing to the compatibility of the length filtration with
multiplication, we also easily obtain.

\begin{cor} For all $r$ we have $\mcg_A^r \subseteq F^r$.
\end{cor}

The following theorem asserts that this inclusion is an equality.

\begin{thm}\label{thm:main1} For any $A_\infty$-algebra $A$, $\mcg_A^r = F^r$. \end{thm}

\begin{proof} We start this proof by writing down two homotopies of the diagonal $(A, A)$-bimodule

\begin{equation*} h^\pm_{diag} : \cbr A \to \cbr A \end{equation*}

where \begin{eqnarray*} h_{diag}^+ ([\mba | a | \mba^\prime] ) & =
& (-1)^{|\mba| + |a|} [[\mba |a]|e | \mba^\prime ] \\ h_{diag}^-
([\mba | a | \mba^\prime] ) & = & (-1)^{|\mba| } [\mba |e |[a|
\mba^\prime] ]
\end{eqnarray*}

While these maps fail to be bi-comodule morphisms, they do lie in
$Map^{1, 1} (\Diagonal_A , \Diagonal_A )$. Indeed, we have

\begin{equation*} [\Delta , h_{diag}^+] ([\mba | a | \mba^\prime]) = [\mba | a] \otimes [e] \otimes
[\mba^\prime ] \end{equation*}  and \begin{equation*} [\Delta ,
h_{diag}^- ] ([\mba | a | \mba^\prime ]) = [\mba ] \otimes [e]
\otimes [a | \mba^\prime ]. \end{equation*}

Furthermore, letting $\tau^\pm$ be the translation maps
\begin{eqnarray*} \tau_+ ( [\mba|a|[a_1^\prime |\cdots
|a_m^\prime]) & = & (-1)^{1 + |\mba| +
|a|}[[\mba |a]|a_1^\prime |[a_2^\prime| \cdots |a_m^\prime]] , \\
\tau_- ( [[a_1| \cdots|a_n]|a|[\mba^\prime]) & = & (-1)^{1 + |a_1|
+ \cdots + |a_{n - 1}|} [[a_1| \cdots|a_{n -
1}]|a_n|[a|\mba^\prime]] ,
\end{eqnarray*} our homotopies bound to

\begin{eqnarray*} \partial h_{diag}^\pm & = & 1 - \tau_\pm  \end{eqnarray*}

More generally, we have

\begin{eqnarray*} \partial \left[ h_{diag}^\pm ( 1 + \tau_\pm + \tau_\pm^2 + \cdots + \tau_\pm^{k - 1} ) \right] & = & 1 - \tau_\pm^{k }, \end{eqnarray*} and by \ref{lem:map1},

\begin{eqnarray*} \sigma^\pm_k :=  h_{diag}^\pm ( 1 + \tau_\pm + \tau_\pm^2 + \cdots + \tau_\pm^{k - 1}
) \in Map^{k,0}_{(\acbr A , \acbr A)} (\cbr \Diagonal_A,
\cbr \Diagonal_A)  \end{eqnarray*}

One observes that for any $l$, as a map in $\co$ the translation
map satisfies
\begin{equation} \label{eq:trans} \tau_-^k (\cbn{[k , l]} \Diagonal_A ) = 0 \end{equation}

 We now use induction to prove our theorem. It suffices to show that if
 $\phi \in \mcf^r \Hom_{\cmd} (M, N)$, then there exists a module $K$ and homomorphisms $\pi : M \to K$, $\psi : K \to N$ such that
 $\pi \in F^1$, $\psi \in F^{r - 1}$ and $\phi = \psi \circ \pi$. We consider the diagram below which is commutative up to homotopy.
\begin{figure}[h] \begin{tikzcd} M \arrow{r}{\epsilon_{l, M}} \arrow{d}[swap]{\phi} \arrow{rd}{\pi} & \Diagonal_A \ts{} M \arrow{d}[swap]{\pi_0} \\
N & \Diagonal_A \odot_1 M \arrow{l}{\psi}
\end{tikzcd}
\end{figure}
The map $\epsilon_{l, M}$ was defined in equation \ref{eq:units}
and is a quasi-isomorphism. In particular, a simple examination of
the map shows that $\Diagonal_A \ts{} M$ has length $1$ as a filtered
module. Thus, a basic application of \ref{lem:lift} implies that
$\pi \in F^1$.

On the other hand, as $\psi$ is the restriction of $\xi_{l, M}
\circ (1 \otimes \phi  )$, we can write it out concretely. It is a
strict map whose restriction to $A \otimes A[1]^{\otimes n}
\otimes M$ is
\begin{equation*} \psi^0_n ([a | a_1 | \cdots | a_n | m]) = \sum_{i = 0}^n (-1)^{|a_1| + \cdots + |a_i|} \mu_N^{i + 1} ([a | a_1| \cdots |a_{i} |  \phi^{n - i}([a_i| \cdots |a_n |m])] ) \end{equation*} 
for $n > 1$. As $\phi \in F^r$, we see in
particular that $\psi^0_n = 0$ for $n \leq r$. Thus $\psi $
factors as a composition 
\begin{equation*} A \odot_1 M
\stackrel{\pi }{\longrightarrow} A \odot_r M
\stackrel{\tilde{\psi} }{\longrightarrow} N
\end{equation*} 
where $\tilde{\psi}$ is a strict homomorphism. Now, a direct calculation shows
that $\sigma_{r - 1} \otimes 1_M : \cbr \Diagonal_A
\otimes_{co} \cbr M \to \cbr \Diagonal_A \otimes_{co} \cbr M$ restricts to a
well defined $A_\infty$-module morphism
\begin{equation*} \sigma^-_{r - 1} \odot_1 1_M : A \odot_1 M
\to A \odot_r M . \end{equation*} Composing with $\tilde{\psi}$ and
applying the differential gives \begin{eqnarray*} \partial [
(-1)^{|\psi |} \tilde{\psi} \circ  (\sigma^-_{r - 1} \odot_1 1_M )
] & = & \tilde{\psi} \circ  ((\partial \sigma^-_{r - 1} ) \odot_1
1_M ) \\ & = & \tilde{\psi} \circ ((1_A - \tau_-^{r - 1} ) \odot_1
1_M )
\\ & = & \tilde{\psi} \circ (1_A \odot_1 1_M ) - \tilde{\psi }
\circ (\tau_-^{r - 1} \odot_1 1_M ) \\ & = & \tilde{\psi} \circ
\pi - \tilde{\psi } \circ (\tau_-^{r - 1} \odot_1 1_M ) \\ & = &
\psi - \tilde{\psi } \circ (\tau_-^{r - 1} \odot_1 1_M ).
\end{eqnarray*} Thus $\psi$ is cohomologous to $\tilde{\psi } \circ (\tau_-^{r - 1}
\odot_1 1_M )$. Yet by \ref{eq:trans} we have that
\begin{equation*} ( \tau_-^{r - 1} \odot_1 1_M ) \left(\cbn{[r -
1, 0]} \Diagonal_A \otimes_{co} \cbr M \right) = 0 \end{equation*} and since
$\tilde{\psi}$ is strict, this implies that \begin{equation*}
\tilde{\psi } \circ ( \tau_-^{r - 1} \odot_1 1_M ) \left(\cbn{[r -
1, 0]} \Diagonal_A \otimes_{co} \cbr M \right) = 0 \end{equation*} Thus,
$\psi \simeq \tilde{\psi } \circ ( \tau_-^r \odot_1 1_M ) \in F^{r
- 1}$. \end{proof}
Combining this theorem with the Ghost Lemma of the previous
section, we have the following homological criteria for generation
time. 

\begin{cor} Given an $A_\infty$-algebra $A$, the generation time of an $A_\infty$-module $\Unit_A$ in
$H^0 (A\mhyphen \cmd)$ is the global length $\ell (A\mhyphen \pcmd )$. \end{cor}
Coupling this to the theory of enhanced triangulated categories, we also obtain the corollary below.

\begin{cor} If $\mca$ is a pretriangulated $A_\infty$-category and
$G \in \mca$ is a generator, then $t (G) = \ell (A_G\mhyphen \cmd )$. \end{cor}

More refined statements on the level $\lvl_G (M)$ of an object with
respect to a given generator $G$ are also of use. We write the
result in the $A_\infty$-module category as opposed to
concentrating on the $A_G$-module case.

\begin{cor} If $M$ is an $A$-module then $\textnormal{lvl}_{A} (M) = \ell (M)$.
\end{cor}

\begin{eg} As was mentioned at the end of Section \ref{subsub:lf}, when $A_G$ is an ordinary algebra, the global length of $\ell (A_G \mhyphen \cmd )$ is precisely its homological dimension. For the cases of the Beilinson exceptional collection $\langle \mco , \ldots ,\mco (n) \rangle$, one may use Beilinson's resolution of the diagonal to see that this dimension is $n$.  
\end{eg}

\begin{eg} For the generator $\mco \oplus \mco_p$ of $\p^1$, we again have formality, but $A_{G^\prime}$ is now a graded algebra. 
Nonetheless, the graded simple modules $S_1$ and $S_2$ arise from considering the idempotents at the vertices and the graded projective modules $P_1$, $P_2$ from considering all arrows mapping out of each vertex. The projective resolutions below for the simple objects give the homological dimension of $A_{G^\prime}$ as $2$.
\begin{equation*} \cdots 0 \to P_1 \to P_2 \to P_1 \to S_1 \to 0 \end{equation*}
\begin{equation*} \cdots 0 \to P_1 \to P_2 \to S_2 \to 0 \end{equation*}
\end{eg}

The final example explores a case where higher products have a significant effect on generation time.

\begin{eg} From example \ref{eg:3}, we recalled that $MF (\C [[z]] , z^n)$ had a generator $G$ with $A_G = k[\theta] / (\theta^2)$ with a single higher product $\mu^n (\theta , \cdots \theta ) = 1$. To describe $H^0 (A_G \mhyphen \pcmd)$, we examine the $A_\infty$-relation for the products of a minimal $A_G$-module $M$. First, we recall that $M$ is $\Z / 2\Z$ graded and the usual $A_\infty$-module map $\mu_M^r : A_G^r \otimes M \to M$ is degree $r + 1$ (due to the desuspension of $A_G$). Since we assume $M$ is unital, $\mu_M^r$ is completely determined by $\mu_M^r ([\theta| \cdots | \theta | m])$. Writing $L_r = \mu_M^r ([\theta | \cdots | \theta | \_ ]) \in \Hom_\gr^1 (M , M)$, we may condense $\mu_M$ into a power series $\mathbf{L} = \sum_{r = 1}^\infty L_r u^r \in \Hom_\K (M , M ) \otimes \C[[u]]$. It is easy to see that the $A_\infty$-relation on $\mu^r_M$ translates into the equality
\begin{equation*} 
 \mathbf{L} \cdot \mathbf{L} = 1_M \cdot u^n
\end{equation*}
Taking $M = M_0 \oplus M_1$, we may decompose $L_r = L_r^0 \oplus L_r^1$ where $L_r^0 : M_0 \to M_1$ and $L_r^1 :M_1 \to M_0$. Summing, we write $\mathbf{L}^i = \sum_{r = 1}^\infty L_r^i u^r$ and after tensoring $M$ with $\C[[u]]$ we then have
\begin{equation*}
 \begin{tikzcd} M_0 \otimes \C[[u]]
\rar[-to, to path={([yshift=0.5ex]\tikztostart.east) -- ([yshift=0.5ex]\tikztotarget.west) \tikztonodes}]{\mathbf{L}^0}
\rar[to-, to path={([yshift=-0.5ex]\tikztostart.east) -- ([yshift=-0.5ex]\tikztotarget.west) \tikztonodes}][swap]{\mathbf{L}^1} 
& M_1\otimes \C[[u]]  \end{tikzcd}
\end{equation*}
with $\mathbf{L}^0 \mathbf{L}^1 = u^n = \mathbf{L}^1 \mathbf{L}^0$. This returns us full circle to the setting of matrix factorizations, but with the added presence of the length filtration. Indeed, as above, given another $A_G$ module $(N, \tilde{\mathbf{L}})$ we may write any morphism $\phi : M \to N$ as a power series $\mathbf{T} = \sum_{r = 0}^\infty T_r u^r \in \Hom^*_\gr (M, N) \otimes \C[[u]]$ where $T_r (m) = \phi ([\theta | \cdots |\theta | m] )$. The differential on $\Hom_\pcmd^* (M, N)$ is the usual matrix factorization differential $d \mathbf{T} = \tilde{\mathbf{L}} \mathbf{T} - (-1)^{|T|} \mathbf{T} \mathbf{L}$. It is obvious from this representation that $\phi \in \mcf^k \Hom_\pcmd (M, N)$ if and only if $\deg (\mathbf{T}) \geq k$. 

For $1 \leq m \leq  \lfloor \frac{n}{2} \rfloor$, and define $M_m$ to be the module corresponding to 
\begin{equation*}
\begin{tikzcd} \C[[u]]
\rar[-to, to path={([yshift=0.5ex]\tikztostart.east) -- ([yshift=0.5ex]\tikztotarget.west) \tikztonodes}]{u^m}
\rar[to-, to path={([yshift=-0.5ex]\tikztostart.east) -- ([yshift=-0.5ex]\tikztotarget.west) \tikztonodes}][swap]{u^{n - m}} 
&  \C[[u]]  \end{tikzcd}.
\end{equation*}
These make up the irreducible modules. It is not hard to show that the maximal filtered homomorphism between any two such modules is $\phi : M_m  \to M_m$ , for $m = \lfloor \frac{n}{2} \rfloor$,  and $\phi$ corresponding to $\mathbf{T} = u^{m - 1}$. This implies the generation time of $G$ is $\deg (\mathbf{T} ) = \lfloor \frac{n}{2} \rfloor - 1$ in agreement with results in \cite{BFK}.

\end{eg}

The last example raises interesting questions on which filtrations arise as length filtrations on the category of matrix factorizations. In the above example, we obtained the $\mathbf{m}$-adic filtration on matrices by considering the generator $R / \mathbf{m}$ where $R = \C[[u]]$ and $\mathbf{m} = (u)$. It is natural to ask when the $I$-adic filtration on morphisms between matrix factorizations arises as a length filtration associated to the generator $R / I$ on the quasi-equivalent derived category of singularities.



\subsection{Change of base formula}

In this subsection we generalize the classical change of base formula for dimension to the case of dimensions of $A_\infty$-algebras. We see that a new multiplicative factor appears in this formula that measures the formality of the algebras involved. 

We start by obtaining a general lemma on filtered
$A_\infty$-modules. To simplify the exposition and some proofs, we
will work with modules as opposed to polymodules. Suppose $M$ is an
$A$-module and $(N, \mcg^*) \in (A\mhyphen \pcmd)^{f}$ is a filtered $A$-module of finite
filtration length and $\phi \in \Hom_{A\mhyphen \cmd} (M, N)$ any map. We wish to
obtain a finite approximation of $\phi$ relative to both the
internal filtration on $N$ and the filtration on $\Hom_{A\mhyphen \cmd} (M,
N)$. A surprising parameter that emerges in this pursuit is the
degeneration time of the spectral sequence associated to $(N,
\mcg^*)$. For the following lemma, assume $\mcg^{-1} N = 0 \ne
\mcg^0 N$, let $N_t = N / \mcg^t N$ and $\pi_t : N \to N_t$ be the
projection.

\begin{lem}\label{lem:lift} Suppose the spectral sequence associated to $(N, \mcg^*)$ degenerates on the $(s + 1)$-page and $\ell (N) = n$.
Then for every $p$ there exists a lift $\gamma$ such that the
following diagram commutes up to homotopy
\begin{figure}[h]
\begin{tikzcd} \text{  } & F^{p + 1} \Hom_{A\mhyphen \cmd} (M, N_{n + sp}) \arrow[hook]{d} \\ 
\Hom_{A\mhyphen \cmd} (M, N)  \arrow[dashed]{ru}{\gamma} \arrow{r}{(\pi_{n + sp + 1})_*} &  \Hom_{A\mhyphen \cmd} (M,  N_{n +
sp} )
\end{tikzcd} 
\end{figure}
\end{lem}

 Before proving this lemma, let us set up some basic
 notation. First we take
 \begin{equation*} \rho_q: \mcf^q \Hom_{\cmd} (M, N) \to \Hom_\co
 ((\cbr M)^q , N) = \Hom_\co (A[1]^{\otimes q} \otimes M, N) \end{equation*}
 to be the restriction map. Here, the right hand side is
 the complex of morphisms from $((B^+ M)^q, b_M^0|_{(B^+ M)^q})$ to $(N, \mu^1_N)$.
 It is worthwhile to note that $\rho_q$ is a map of cochain complexes (i.e. $d \rho = 0$ in $\co$).

Let us also introduce a general ``strictification" map.

\begin{equation*} \sigma_q : \Hom_\co (A[1]^{\otimes q } \otimes M, N) \to \Hom_{\cmd} ( M, N) \end{equation*}

This is the map $\sigma_q (\phi ) = \{ \phi^{ k}\}$ where
\begin{equation*} \phi^{k} = \left\{ \begin{matrix}  \phi &
\text{if } k = q \\ 0 & \text{otherwise } \end{matrix}
\right.\end{equation*} We note that this is not in general a
cochain complex map. Nevertheless, it is clear that, for every
$q$, \begin{equation} \label{eq:st1} \rho_q \circ \sigma_q = 1
\end{equation}

\begin{proof} We start by proving the following claim. \\

 \textit{Claim}: With the assumptions of the lemma, for every $q$, the following diagram in $\ho$ has a lift which commutes up to homotopy,
 \begin{figure}[h]
 \begin{tikzcd} \text{  } &  \mcf^{q } \Hom_{\cmd} (M, N_{n}) \arrow[hook]{d} \\ 
\mcf^q \Hom_{\cmd} (M, N) \arrow[dashed]{ru}{\alpha_q} \arrow{r}{(\pi_{n})_*} &  \Hom_{\cmd} (M,  N_{n } )
\end{tikzcd}
 \end{figure}
 such that $\alpha_q (\Hom_{\cmd} (M, N)) \subseteq \mcf^{q + 1} \Hom_{\cmd} (M, N_n)$. \\

 We first observe that since $\mcg^{-1} N = 0 \ne \mcg^0 N$ and $\ell (N) = n$ the map $\pi_n : N \to N_t$ is contractible.
 Thus, restricting to $(\cbr M)^q$, the induced map on chain complexes
 \begin{equation*} (\tilde{\pi}_n)_* : \Hom_\co^* (A[1]^{\otimes q } \otimes M, N) \to \Hom_\co^* (A[1]^{\otimes q} \otimes M, N_t) \end{equation*}
 is also contractible. Here the differential associated to $A[1]^{\otimes q } \otimes M$ is the restriction of $b_M^0$.

 We use the notation of $\tilde{\pi}_n$ above in order to distinguish it from the map in the claim, but both are obtained through composition and the equation
 \begin{equation} \label{eq:st2} \rho_q \circ (\pi_n)_* = (\tilde{\pi}_n)_* \circ \rho_q \end{equation} holds.
 Let \begin{equation*} \tau : \Hom_\co^* (A[1]^{\otimes q } \otimes M, N) \to \Hom^{* - 1}_\co (A[1]^{\otimes q } \otimes M, N_t) \end{equation*}
 be a cochain bounding $(\tilde{\pi}_n)_*$ (i.e. $(\tilde{\pi}_n)_* = d \tau$ in $\co$) and
 take \begin{equation*} \alpha_n (\phi ) = [(\pi_n)_*  - d(\sigma_q \circ \tau \circ \rho_q)] ( \phi ) \end{equation*}
 Observe that this is a map in $\ho$ by virtue of $(\pi_n)_*$
 being a cochain map and the fact that $d f$ is cochain map for
 any $f$ in $\co$. It is equally obvious that the diagram then commutes
 up to homotopy. So the only point left to prove for the claim is that any
 module homomorphism $\phi \in \Hom_{\cmd} (M, N)$ must have image in $\mcf^{q +
 1} \Hom_{\cmd} (M, N_n)$. This is true iff $\rho_q (\alpha_n (\phi
 )) = 0$. Since $\rho_q$ is a chain map, we have $\rho_q (d g) = d
 (\rho_q (g))$, and by \ref{eq:st1}, \ref{eq:st2} \begin{eqnarray*} \rho_q
 (\alpha_n (\phi )) & = & \rho_q ([(\pi_n)_*  - d(\sigma_q \circ \tau \circ \rho_q)] ( \phi
 )) \\ & = & \rho_q \circ (\pi_n)_* (\phi) - \rho_q [d(\sigma_q \circ \tau \circ \rho_q) ( \phi
 )] \\ & = & (\tilde{\pi}_n)_* \circ \rho_q (\phi) - d[ (\rho_q  \circ \sigma_q \circ \tau \circ \rho_q) ( \phi
 ))] \\ & = & (\tilde{\pi}_n)_* \circ \rho_q (\phi) - d[ (\tau \circ \rho_q) ( \phi
 )] \\ & = & ( \tilde{\pi}_n)_* \circ \rho_q (\phi) - (d\tau ) \circ \rho_q ( \phi
 ) \\ & = & (\tilde{\pi}_n)_* \circ \rho_q (\phi) - (\tilde{\pi}_n)_* \circ \rho_q ( \phi
 ) \\ & = & 0 \end{eqnarray*}

One now uses the claim to prove the lemma by observing that if
$(C^*, \mcg)$ is any filtered chain complex whose length is $r$
and whose spectral sequence converges at the $(p + 1)$-th page,
then $\ell (C / \mcg^r C, \mcg ) \leq p$. This argument relies on
simply unravelling the definition of the spectral sequence
associated to a filtration. We recall that the page $E_k^q = Z_k^q
/ B_k^q$ is the subquotient of $\mcg^k C / \mcg^{k - 1} C$ where
\begin{equation*} Z_k^q = \{[c]: c \in \mcg^k C , dc \in \mcg^{k -
q} C \} \end{equation*} and
\begin{equation*} B_k^q = \{[dc] : c \in \mcg^{k + q - 1} C, dc
\in \mcg^{k } C \} \end{equation*} Note then that $E_k^{r + q}$ is
the same as $\tilde{E}_k^q$ for $q > p$ where the later is the
spectral sequence for $(C / \mcg^r C, \mcg^{* - r} )$. In
particular, $\tilde{E}_k^q = 0$ for all $q > p$ implying the
length $\ell (C / \mcg^r C, \mcg ) \leq p$. To finish the proof,
just inductively apply the claim above and this observation with
$(N, \mcg^*)$.

 \end{proof}

The following theorem is a result of \ref{lem:lift}.

\begin{thm} \label{thm:main2} Let $P$ be a $(B, A)$-bimodule and $M$ a left $A$-module. Suppose the
spectral sequence of $P \ts{A} M$ degenerates at the $(s + 1)$-st
page. If the convolution functor $P \ts{} \_$ is faithful, then

\begin{equation*} \textnormal{lvl}_A (M) \leq \textnormal{lvl}_A (P ) + s \cdot \textnormal{lvl}_B (P
\ts{A} M ) \end{equation*} \end{thm}

\begin{proof} Assume that this is not the case. Then there exists
a nonzero morphism $f \in F^r \Hom^* (M, N)$ with $r > \lvl_A
P^\vee + s \lvl_B (P \ts{A} M )$. Then by definition, $1_P \ts{} f$
factors through $P \tns{r - 1 } M$ implying $1_P \ts{} f = \psi
\circ \pi_r = \pi_r^* (\psi )$ where $\pi_r : P \ts{} M \to P
\odot_r M$. Now, by assumption, the spectral sequence associated
to $P \ts{A} M$ degenerates at $(s + 1)$ and by \ref{cor:prequalsfl},
the $\ell (P \ts{A} M) \leq \ell ( P^\vee ) = \lvl_A (P )$. Letting
$n = \lvl_A P$, the following lifting problem is solvable for
all $p$ by \ref{lem:lift}
\begin{figure}[h] 
\begin{tikzcd} \text{  } & F^{p + 1} \Hom_{\cmd} (P \ts{} M , P \odot_{n + sp + 1} M) \arrow[hook]{d} \\ 
\Hom_{\cmd} (P \ts{} M, P \ts{} M) \arrow[dashed]{ru}{\gamma} \arrow{r}{(\pi_{n + sp + 1})_*} &  \Hom_{\cmd} (P \ts{} M, P \odot_{n + sp  +1} M)
\end{tikzcd} 
\end{figure}

In particular, if $p = \lvl_B (P \ts{A} M )$ we have that $\pi_{n +
sp + 1} \simeq 0$. This implies that for all 
\begin{equation*} t \geq  n + sp + 1 = \lvl_A P^\vee + s \cdot \lvl_B (P \ts{A} M )
\end{equation*} 
we must have $\pi_t \simeq 0$ so that $\pi_r
\simeq 0$ and therefore $1_P \ts{} f \simeq 0$. This contradicts
the assumption that $P \ts{} \_$ is faithful.
\end{proof}


\bibliographystyle{plain}
\bibliography{base}

\end{document}